\documentclass{article}

\usepackage{geometry}
\usepackage{color}
\usepackage{slashed}

\usepackage[numbers,sort&compress]{natbib}
\usepackage{graphicx,latexsym,euscript,makeidx,color,bm}
\usepackage{amsmath,amsfonts,amssymb,amsthm,thmtools,mathtools,mathrsfs,enumerate}
\usepackage[colorlinks,linkcolor=blue,anchorcolor=green,citecolor=red]{hyperref}

\def\5n{\negthinspace \negthinspace \negthinspace \negthinspace \negthinspace }
\def\4n{\negthinspace \negthinspace \negthinspace \negthinspace }
\def\3n{\negthinspace \negthinspace \negthinspace }
\def\2n{\negthinspace \negthinspace }
\def\1n{\negthinspace }

\def\dbE{\mathbb{E}}     
\def\dbF{\mathbb{F}} \def\sF{\mathscr{F}}    
         
\def\dbH{\mathbb{H}}

\def\dbP{\mathbb{P}}     
     
\def\dbR{\mathbb{R}}

\def\Om{\Omega}

\def\ms{\medskip}

\def\no{\noindent}        \def\q{\quad}                      
                        
    \def\hb{\hbox}                     
                   
         \def\rf{\eqref}                    
                 
            \def\({\Big (}
                  \def\){\Big )}
\def\leq{\leqslant}       \def\geq{\geqslant}
\def\ges{\geqslant}          \def\[{\Big[}
           \def\]{\Big]}
                   \def\cd{\cdot}

\def\e{\varepsilon}             
           \def\i{\infty}   

\theoremstyle{plain}

\theoremstyle{rmk}

\makeatletter

\@addtoreset{equation}{section}
\makeatother

\newtheorem{theorem}{Theorem}[section]
\newtheorem{definition}[theorem]{Definition}
\newtheorem{proposition}[theorem]{Proposition}

\newtheorem{lemma}[theorem]{Lemma}
\newtheorem{remark}[theorem]{Remark}

\newtheorem{assumption}{Assumption}

\makeatletter

\@addtoreset{equation}{section}
\makeatother

\allowdisplaybreaks[4]

\def\sqr#1#2{{\vcenter{\vbox{\hrule height.#2pt
				\hbox{\vrule width.#2pt height#1pt \kern#1pt \vrule width.#2pt}
				\hrule height.#2pt}}}}

\begin{document}
	\title{\Large \bf A global stochastic maximum principle for delayed forward-backward stochastic control systems
	}
	
	\author{
		%
		Feng Li\thanks{
			Department of Mathematics,
			Southern University of Science and Technology, Shenzhen, Guangdong, 518055, China
			(Email: {\tt 12331011@mail.sustech.edu.cn}).}~~~~
		%
	}
	\date{}
	\maketitle
	\no\bf Abstract. \rm
	In this paper, we study a delayed forward-backward stochastic control system in which all the coefficients depend on the state and control terms, and the control domain is an arbitrary nonempty set. 
	To handle the terms involving $y$ and $z$ in the second-order expansion, we derive two decompositions of $\big(\hat{y}^\e(\cd), \hat{z}^\e(\cd)\big)$. More precisely, we use the second decomposition to handle the terms involving $f^\e_{yy}, f^\e_{yz}$ and use both decompositions to handle the term involving $f^\e_{zz}$. Finally, a global stochastic maximum principle is obtained.
%

	%
	\ms
	
	\no\bf Key words:  
	\rm 
	Delayed forward-backward stochastic control systems; nonconvex control domain; global stochastic maximum principle.
	\ms
	
	\no\bf AMS subject classifications. \rm 93E20, 60H20, 49K45.
	\section{Introduction}

		Let $\{W_t; t\ge0\} $ be a $d$-dimensional standard Brownian motion defined on some complete probability space $(\Om,\sF,\dbP)$,  and let $\dbF \triangleq \{\sF_{t}\}_{t \ges 0}$ be the natural filtration of ${W}$ augmented by all the $\dbP$-null sets in $\sF$. 
	Let  $T\in (0, +\i)$ be  a fixed terminal time,  $\delta \in (0, T)$  a fixed  time delay,   $\kappa \in \dbR$    a fixed constant, and $U\subset \dbR^j$ a nonempty subset.
	In this paper, we consider the following  delayed forward-backward stochastic control system:
\begin{equation}\label{2.1}\left\{\begin{aligned}
		dx(t) & = b\(t, x(t), x(t-\delta), \int_{-\delta}^{0} e^{\kappa \theta}x(t+\theta)d\theta, u(t), u(t-\delta)\)dt \\
		&\q  + \sigma\(t, x(t), x(t-\delta), \int_{-\delta}^{0} e^{\kappa \theta}x(t+\theta)d\theta, u(t), u(t-\delta) \)dW(t),\q t\in [0,T];\\
		dy(t)& = -f\(t, x(t), x(t-\delta), \int_{-\delta}^{0} e^{\kappa \theta}x(t+\theta)d\theta, y(t),  z(t),%
		u(t), u(t-\delta)\)dt  +  z(t) dW(t),  \ \  t\in [0,T]; \\
		x(t) &=\xi(t), \q 
		u(t) = \gamma(t), \q t \in [-\delta, 0];\\
		y(T) &= h\(x(T), x(T-\delta), \int_{-\delta}^{0} e^{\kappa \theta}x(T+\theta)d\theta\),
	\end{aligned}\right.\end{equation}
where  $\xi(\cdot)$ and $\gamma(\cd)$  are given deterministic functions,
$b: [0, T]\times\dbR^n\times \dbR^n \times \dbR^n \times U \times U \rightarrow \dbR^n$, 
$\sigma: [0, T]\times\dbR^n\times \dbR^n \times \dbR^n\times U \times U \rightarrow \dbR^{n\times d}$,
$h: \dbR^n \times \dbR^n \times \dbR^n \rightarrow \dbR$, 
and 
$f: [0, T]\times \dbR^n\times \dbR^n\times \dbR^n \times\dbR\times\dbR^{d}\times U \times U \rightarrow \dbR$ are some functions,  
	and the control process $u(\cd)$ comes from the following admissible control set:
	\begin{equation}
		\begin{aligned}
			\mathcal{U}_{ad} \triangleq \Big\{u(\cd) : [-\delta, T] \rightarrow \dbR^j \ \big|\ & \hb{$u(\cd)$  is a  $U\hb{-valued}$, $\dbF$-progressively measurable process,  and satisfies} \   \\  &
			u(t) = \gamma(t), \ \  t\in [-\delta, 0] \ \ \hb{and} \ \
			\dbE[\mathop{\sup}\limits_{t\in [-\delta, T]}|u(t)|^\beta]<\i, \ \ \forall \beta\geq 1 \Big\}.  
	\end{aligned}\end{equation}
	
	{{\it \bf Problem (O)}}. Find a control $u^*(\cd)$ over $\mathcal{U}_{ad}$ such that the system  \eqref{2.1} is satisfied and  the cost functional $J(u(\cdot))\triangleq y(0)$ is minimized, i.e.,
	\begin{align}\label{problem}
		J(u^*(\cd)) = \mathop{\inf}\limits_{u(\cdot)\in \mathcal{U}_{ad}}J(u(\cdot)).
	\end{align}
	%
	%
	%
	%
	%
	%
	Moreover, we refer to system \rf{2.1} as a \emph{delayed forward stochastic control system} if it is independent of $y$ and $z$.
	Additionally, the system \rf{2.1} is called a \emph{forward-backward stochastic control system} if it is independent of all delay  terms.
	On one hand, the problem of forward-backward stochastic control systems in which the control domain is not necessarily convex was listed as an open problem by Peng \cite{peng1999open}.
	On the other hand, stochastic differential equations (SDEs, for short) typically model processes where state changes depend  on the current state. In many practical situations, however, the state evolution also depends on past states, making stochastic differential delay equations (SDDEs, for short) a more appropriate modeling framework (see, for example, the applications of SDDEs in  \cite{Chen2010delaySDDE, Shenzeng2014, Xushenzhang2018}). Therefore, it is both natural and meaningful to extend the forward-backward stochastic control systems to the more general framework described by system \eqref{2.1}. 
	So far, there has been extensive literature that addresses delayed forward stochastic control systems and forward-backward stochastic control systems.
	%
	%

	For delayed forward stochastic control systems, to overcome the main difficulty arising from the delay term, in 2009, Peng and Yang \cite{Peng2009anticipated} introduced the anticipated backward stochastic differential equations (ABSDEs, for short) as a dual process to the SDDEs.
	Since then, by applying this approach,  a lot of researchers have studied the delayed forward stochastic control systems under various cases.  For example,  Meng and Shi \cite{Meng2021} addressed the stochastic optimal control problem with delay where the control domain is not necessarily convex. However, to eliminate the cross terms of states and delay states, the solution of some second-order adjoint equations in Meng and Shi \cite{Meng2021} is required to be identically zero. Some other literature on using the duality between SDDEs and ABSDEs can be found in  \cite{oksendal2011optimal, Chen2010delaySDDE, anticipate 1, anticipate 2, anticipate 3}, etc.
	Recently, Meng, Shi, Wang, and Zhang  \cite{meng2025general} obtained a general maximum principle  without the strong condition imposed in Meng and Shi \cite{Meng2021} by introducing a new method---transforming the delayed variational equations into stochastic Volterra integral equations (SVIEs, for short), and then using the duality between SVIEs and backward stochastic Volterra integral equations (BSVIEs, for short). In addition, Guatteri and Masiero \cite{guatteri2025stochastic} transformed the SDDE into an infinite dimensional stochastic evolution equation and then  proceeded similarly to Fuhrman, Hu, and Tessitore \cite{Fuhrman 2012} and Fuhrman, Hu, and Tessitore\cite{Fuhrman 2013}.

	For forward-backward stochastic control systems, 
	when the control domain is convex,  some related works can be found in Dokuchaev and Zhou \cite{convex 1}, Ji and Zhou \cite{convex 2},  Wu \cite{convex 3}, and Lv, Tao, and Wu \cite{lv2016maximum}, etc. 
	When the control domain is not necessarily convex, it was, in fact, an open problem listed in Peng \cite{peng1999open} and has attracted considerable attention.
	For instance, 
	Shi and Wu \cite{Shi2006} derived a global maximum principle  under the assumption that the diffusion term of the forward equation  is  independent of the control process.
	%
	Subsequently, Hu \cite{Hu} 
	applied It\^{o}'s formula to derive first-order and second-order adjoint equations and solved  the open problem of Peng \cite[Section 4]{peng1999open} completely.  
	Hu \cite{Hu}'s method has been successfully extended to more general systems. For example, Hu, Ji, and Xue \cite{hu2018global} and Hu, Ji, and Xu \cite{hu2022global} derived the global maximum principle for fully coupled forward-backward stochastic control system and with quadratic generator, respectively. 
	Hao and Meng \cite{hao2020global} and Buchdahn, Li, Li, and Wang \cite{buckdahn2024global} derived a global maximum principle for  mean-field forward-backward control system with linear growth and quadratic growth generator, respectively.

	%
	%
	%
	%

	However,  for Problem (O),  to the best of our knowledge, there have been few results  so far. 
	The only related work is that of Huang and Shi \cite{huang2012maximum}, in which a general maximum principle was obtained under the assumption that the diffusion term is  independent of the control process. 
	Due to the absence of the control process in the diffusion term, Huang and Shi \cite{huang2012maximum} only needed to deal with the first-order variational equation of the state equation, which is essentially analogous to the case of a convex domain. 
	%
	In this paper, we will study the general system \rf{2.1}, in which the diffusion term in  the forward equation contains both the state and control terms, and aim to obtain a general maximum principle. 
	In contrast to Huang and Shi \cite{huang2012maximum}, we need to handle both the first-order and second-order variational equations.  As a result, problem (O) is more challenging. 
	
	To solve  {\bf Problem (O)}, we first establish a rough and a precise decompositions for $(\hat{y}^\e, \hat{z}^\e)$ (see \eqref{decom}, \eqref{dey}, and \eqref{de} in \autoref{dec}). Since the classical It$\hat{\rm o}$'s  formula does not generally apply directly to solutions of SVIEs and BSVIEs, establishing their duality relationship requires introducing auxiliary stochastic processes to which the It$\hat{\rm o}$'s formula applies.  This makes their duality less transparent than that between SDEs and BSDEs. We therefore characterize some components of our decompositions using BSDEs. 
	Second, using the precise decomposition and arguments similar to those of Meng, Shi, Wang, and Zhang \cite{meng2025general}, we handle the first-order terms and all second-order terms except the terms involving $f^\e_{yy}e^2, f^\e_{yz}er, f^\e_{zz}r^2$ (see \autoref{first ad}). Third,  we use the precise decomposition  to handle the terms $f^\e_{yy}e^2, f^\e_{yz}er$. After substituting the precise decomposition, we introduce three BSDEs (see \eqref{aad3}) to handle the terms $\mathcal{X}_1(T, s)^\top [...] \mathcal{X}_1(T, s)$, $\mathcal{X}_1(T, s)^\top [...] \mathcal{X}_1(\theta, s)$, and $\mathcal{X}_1(\theta, s)[...]\mathcal{X}_1(k, s)$, thereby obtaining terms of the same form as those in Meng, Shi, Wang, and Zhang \cite[Lemma 4.2]{meng2025general}.  Note that the $L^2_{\dbF}$-integrability of the solutions to the BSDEs \eqref{aad3} follows from the result of Hamaguchi and Taguchi \cite[Theorem 3.4]{hamaguchi2023approximations}.
	Fourth, we use \eqref{ee} to represent $\Gamma^0|r|^2 $ in terms of $\Gamma^0|e|^2, \Gamma^0er$, and then obtain the desired result for  using the preceding analysis of $f^\e_{yy}e^2, f^\e_{yz}er$ (see \autoref{iim}).  Finally, we obtain a global stochastic maximum principle for {\bf Problem (O)}.
	
	The first difficulty lies in establishing the precise decomposition using the duality between SVIEs and BSVIEs. The second difficulty is that, when this decomposition is used to handle the term $f^\e_{zz}r^2$, the BSVIEs introduced for the subsequent duality analysis  do not directly ensure the integrability required. We therefore use the rough decomposition  to handle the term involving $\frac{1}{2}f^\e_{zz}r^2$. More precisely, 
 Hamaguchi and Taguchi \cite[Theorem 3.4]{hamaguchi2023approximations}  only guarantee that $\dbE[\big(\int_{0}^{T}\int_{0}^{T}\int_{0}^{T}|\zeta^0(\theta, s)\zeta^0(k, s)|^2d\theta dk ds\big)^2]<\i$ (see \autoref{aua}), not $\dbE[\int_{0}^{T}\big(\int_{0}^{T}\int_{0}^{T}|\zeta^0(\theta, s)\zeta(k, s)|^2d\theta dk\big)^2 ds]<\i$. Consequently, we cannot handle the terms $f^\e_{zz}r^2$ in the same way as $f^\e_{yy}e^2, f^\e_{yz}er$. 
The innovations and contributions of this paper are as follows:
	\begin{itemize}
		\item [$\rm (i)$] 
		We establish a global stochastic maximum principle for delayed FBSDE control system in which all coefficients depend on the state and control variables and the control domain is not necessarily convex.
		\item[$\rm (ii)$] We establish two decompositions, one rough and one
		precise,  to handle the terms involving $f^\e_{xy}(s)$, $ f^\e_{x_\delta y}(s)$, $f^\e_{\tilde{x}y}(s)$, $f^\e_{xz} (s)$, $f^\e_{x_\delta z}(s)$, $f^\e_{\tilde{x}z}(s)$, $f^\e_{yy}(s)$, $f^\e_{yz}(s)$, $f^\e_{zz}(s)$, 
		which arise from the second-order Taylor expansion of $V_2(s)$ in \eqref{v1}. Since the rough decomposition does not involve BSVIE duality, it ensures that the adjoint equations introduced to handle the term of $f^\e_{zz}(s)|r(s)|^2$ satisfy the required integrability conditions. 
		\item [$\rm (iii)$] 
		We provide a new method to handle the  FBSDE control system. The term $V_1(s)$ in \eqref{v1} differs from the corresponding term $B_1(s) - [f(s, \bar{x}(s), \bar{y}(s), \bar{z}(s)+p(s)\delta\sigma(s), u(s))- f(s)]1_{E_\e}(s)$ in Hu \cite{Hu}. This allows us to use the rough and the precise decompositions to handle the term involving $f^\e_{zz}(s)$ arising in the second-order expansion of $V_2(s)$.

		
	\end{itemize}

	The rest of the paper is organized as follows. \autoref{section2} introduces the key notation and several preliminary results used later. In \autoref{Section3}, we present the assumptions for system \eqref{2.1}. In \autoref{dec}, we establish two decompositions for $(\hat{y}^\e, \hat{z}^\e)$: the rough decomposition and the precise decomposition.  Subsequently, in \autoref{first ad}, we handle the term $X_1^\top FX_1$, $eF_{1y}X_1, rF_{1z}X_1 $ by using the second decomposition and the duality relationship between SVIEs and BSVIEs.  In \autoref{section6}, we handle the terms $\frac{1}{2}f_{yy}e^2, f_{yz}er, \frac{1}{2}f_{zz}r^2$ by using the rough and precise decompositions. Finally, \autoref{section7} gives a global stochastic maximum principle for {\bf Problem (O)}.

	\section{Preliminaries}\label{section2}
	Let $(W_t)_{t\in[0, +\i)}$ be a $d$-dimensional standard Brownian motion defined on some complete probability space $(\Om,\sF,\dbP)$ and $\dbF \triangleq \{\sF_{t}\}_{t \ges 0}$ is the natural filtration of ${W}$ augmented by all the $\dbP$-null sets in $\sF$.
	The notation $\mathbb{R}^{m \times d}$ denotes the space of $m \times d$-matrix $M$, equipped with the Euclidean norm defined as $|M| = \sqrt{\text{tr}(MM^{\top})}$, where $M^{\top}$ is the transpose of $M$.
	Denote $m(\cd)$ the Lebesgue measure.
	For any given $t\in [0, T]$, $\beta, \beta' \geq 1$ and Euclidean space $\dbH$, we introduce the following spaces:
	\begin{align*}
		&S^\beta_{\dbF}(t,T;\dbH)
		= \Big\{Y:\Om\times [t,T] \to\dbH\Bigm| Y~ \hb{is $\dbF$-progressively measurable, continuous, and~}  \\
		& \quad \quad \quad \quad \quad \quad\quad\ \quad \quad \quad   \|Y\|_{S^\beta_{\dbF}(t,T)} := \Big(\dbE[\mathop{\sup}\limits_{s \in [t,T]} |Y_{s}|^\beta]\Big)^\frac{1}{\beta}<\i \Big\}, \\
		%
		%
		&L^\beta_{\dbF}(t,T;\dbH)
		= \Big\{Y:\Om\times [t,T] \to\dbH\Bigm| Y~ \hb{is $\dbF$-progressively measurable, 
			and~}    \dbE\Big[\int_{t}^{T}|Y_s|^\beta ds\Big]^{\frac{1}{\beta}}<\i \Big\},\\
		&	\mathcal{M}_{\dbF}^{\beta, \beta'}(t,T; \dbH)
		=\bigg\{Y : \Om\times [t,T] \to\dbH \Bigm| Y~ \hb{is $\dbF$-progressively measurable and~} 
\dbE\bigg[\Big(\int_{t}^{T}|Y_s|^\beta ds\Big)^{\frac{\beta'}{\beta}}\bigg]^\frac{1}{\beta'}<\i \bigg\}, \\
		%
		%
		&C_{\dbF}\big([t, T]; L^\beta(\Om; \dbH)\big)
		=\bigg\{\phi : \hb{$\Om\times [t,T] \to\dbH \Bigm| t \mapsto \phi(\cd,t)$ is continuous}, \\
		&\quad \quad \quad \quad \quad \quad\hb{$\phi$  is $\dbF$-progressively measurable, and } \mathop{\sup}\limits_{s \in [t,T]} \Big(\dbE\big[|\phi(s)|^\beta\big]\Big)^{\frac{1}{\beta}}<\i \bigg\}, \\
		&L^2\big(t, T ; L^2_{\dbF}(t, T; \dbH)\big)
		=\bigg\{\phi : \Om\times [t,T]^2 \to\dbH \Bigm| \hb{for almost all $\theta\in [t, T], \phi(\theta,\cd) \in \mathcal{M}^{2,2}_{\dbF}(t, T; \dbH)$, and} 
		\\
		&\quad \quad \quad \quad \quad \quad\quad\ \quad \quad \quad
		\dbE \Big[\int_{t}^{T}\int_{t}^{T} |\phi(\theta,s)|^2dsd\theta\Big]<\i \bigg\},\\
		&
		L^2\big(t, T; L^2\big(t, T ; L^2_{\dbF}(t, T; \dbH)\big) \big) = \bigg\{\phi : \Om\times [t,T]^3 \to\dbH \Bigm| \hb{for almost all $\theta, s\in [t, T],$} 
		\\
		&\quad \quad \quad \quad \quad \quad\quad\ \quad \quad \quad
		\hb{$\phi(\theta, s, \cd) \in \mathcal{M}^{2,2}_{\dbF}(t, T; \dbH)$, and} \q \dbE \Big[\int_{t}^{T}\int_{t}^{T}\int_{t}^{T} |\phi(\theta,s, \theta')|^2dsd\theta d\theta'\Big]<\i \bigg\}.
	\end{align*}
	Denote by $\mathcal{E}(M)$ the Dol\'{e}ans-Dade exponential of a continuous local martingale $M$, that is, $\mathcal{E}(M) \triangleq \exp \{M_t - \frac{1}{2} \langle M\rangle_t\}$ for any $t\in [0, T]$.

	\ms

	Finally, we present the Lebesgue differential theorem, which can be found in Stein--Shakarchi \cite[page 104]{stein2009real}.
	\begin{proposition}\label{Lebes}
		If $l: [0, T] \rightarrow \dbR$ is integrable on $[0, T]$, then we have 
		\begin{align*}
			\lim\limits_{\e \rightarrow 0 } \frac{1}{\e}\int_{t'}^{t'+\e}l(t)dt = l(t'), \q \hb{for a.e. $t'$}.
		\end{align*}
	\end{proposition}
	
	\begin{remark}
		Since \autoref{Lebes} is applied only finitely many times in this paper, and a finite union of null sets is still a null set, we omit explicit mention of these exceptional null sets in subsequent applications of \autoref{Lebes}.
	\end{remark}
\section{The variational equations }\label{Section3}
In this section, we present the variational equations of the SDDE in \eqref{system}, derive the corresponding estimates, and subsequently apply a transformation inspired by Meng--Shi--Wang--Zhang \cite{meng2025general}. For simplicity of presentation, the constant $K$ will change from line to line in the rest of proof.
Denote 
\begin{align*}
x_{\delta}(t) \triangleq x(t-\delta), 
\q \tilde{x}(t) \triangleq \int_{-\delta}^{0} e^{\kappa \theta}x(t+\theta)d\theta, 
\q \mu(t) \triangleq u(t-\delta).
\end{align*}
Then the control system \eqref{2.1} can be rewritten in a more concise form as follows:
\begin{equation}\label{system}\left\{\begin{aligned}
		&dx(t) = b\big(t, x(t), x_\delta(t), \tilde{x}(t), u(t), \mu(t)\big)dt
		+ \sigma(t, x(t), x_\delta(t), \tilde{x}(t), u(t), \mu(t))dW(t),\q t\in [0,T];\\
		&dy(t) = -f\big(t, x(t), x_\delta(t), \tilde{x}(t), y(t), z(t),  
		u(t), \mu(t)\big)dt +  z(t) dW(t),  \quad t\in [0,T]; \\
		&x(t) =\xi(t), \q t\in [-\delta, 0]; 
		\q u(t) = \gamma(t), \q t \in [-\delta, 0];\\
		&y(T) = h\big(x(T), x_\delta(T), \tilde{x}(T)\big).
	\end{aligned}\right.\end{equation}

Now we present the assumption of \eqref{system}.
\begin{assumption}\rm\label{A1}
	\begin{itemize}
		\item [$\rm(i)$] $\xi$ is a deterministic continuous function, $\mathop{\sup}\limits_{t\in [-\delta, 0]}|\gamma(t)|^\beta < \i$,
		$h$ is twice continuously differentiable in $(x, x_\delta, \tilde{x})$, and  $\partial^2h$ is bounded.
		\item [$\rm(ii)$]$b, \sigma$ are twice continuously differentiable in $(x, x_\delta, \tilde{x})$; 
		$\partial b, \partial \sigma, \partial^2b, \partial^2\sigma$ are bounded;
		there exists a constant $L_1>0$ such that 
		\begin{align}\label{b}
			|b(t, 0, 0, 0, u, \mu)| + |\sigma(t, 0, 0, 0, u, \mu)| \leq L_1(1+ |u|+|\mu|).
		\end{align}
		\item [$\rm(iii)$] $f$ is twice continuously differentiable in $(x, x_\delta, \tilde{x}, y, z)$; $\partial f, \partial^2f$ are bounded; there exists a constant $ L_2>0$ such that 
		\begin{align*}
			|f(t, 0, 0, 0, 0, 0, u, \mu)| \leq L_2(1+|u|+|\mu|),
			%
			%
			%
			%
		\end{align*}
		%
	\end{itemize}
	\noindent where $\partial h, \partial b, \partial \sigma$ and $\partial^2h, \partial^2 b, \partial^2 \sigma$ are the gradients and Hessian matrices of $h, b, \sigma$ with respect to $(x,x_\delta, \tilde{x})$, respectively; $\partial f, \partial^2f$ are the gradients and the Hessian matrices of $f$ with respect to $(x,x_\delta, \tilde{x}, y, z)$, respectively.
\end{assumption}
\begin{remark}
It is easy to check that $f(t, x, x_\delta, y, z, u, \mu)=1+x + x_\delta  + y+ z+ \sin(x+x_\delta+y+z+u+\mu)+ \frac{u^2}{1+u^2}\sqrt{1+(x+x_\delta+y+z)^2} + u + \mu$ with $n=j=d=1$ satisfying (iii) of \autoref{A1}.
\end{remark}
%
The following is the well-posedness result of \eqref{system}, which can be found in Mohammed \cite{mohammed1984stochastic} and Briand--Delyon--Hu--Pardoux--Stoica \cite[Theorem 4.2]{briand2003lp}.
\begin{proposition}\label{thm2.1}
Under \autoref{A1}, for any $u(\cdot) \in \mathcal{U}_{ad}$ and $\beta>1$, the equation \eqref{system} has a unique solution $(x, y, z) \in S^\beta_{\dbF}(-\delta,T;\dbR^n) \times S^\beta_{\dbF}(0,T;\dbR) \times \mathcal{M}_{\dbF}^{2, \beta}(0,T; \dbR^d)$.
Moreover, the following estimate holds:

$\|x\|^{\beta}_{S^\beta_{\dbF}(-\delta,T;\dbR^{n})} 
\leq 
K \Big\{\dbE\Big[\|\xi\|_{S^\beta_{\dbF}(-\delta,0;\dbR^{n})}^\beta + \Big(\displaystyle\int_{0}^{T}|b(s, 0, 0, 0,  u(s), \mu(s))|ds\Big)^{\beta}+ \Big(\displaystyle\int_{0}^{T}|\sigma(s, 0,  0, 0, u(s), \mu(s))|^2ds\Big)^{\frac{\beta}{2}}\Big] 
\Big\}$.
\end{proposition}
%

%

Now we analyze the variational equations of SDDE in \eqref{system}, which serve as critical tools for deriving the stochastic maximum principle.

Let $u^*(\cdot)$ be optimal control and $(x^*(\cdot), x_\delta^*(\cd), \tilde{x}^*(\cd), y^*(\cdot),z^*(\cdot))$ be the corresponding state trajectories of \eqref{system}. 
Notice that the control domain $U$ is not necessarily convex, we will use the spike variation method. Let $0<\e < \delta$, for any $u(\cdot) \in \mathcal{U}_{ad}$ and any fixed $t_0 \in [0, T)$, define $u^\e(\cd)$ as follows:
\begin{equation}\label{}u^\e(t)\triangleq \left\{\begin{aligned}
		&u^*(t), \q t  \notin [t_0, t_0 + \e];\\
		&u(t), \q t \in [t_0, t_0+\e].
	\end{aligned}\right.\end{equation}
which is a perturbed admissible control of the form.
In addition, we let $(x^\e(\cdot), y^\e(\cdot), z^\e(\cdot))$ denote the state trajectories of \eqref{system} associated with $u^\e (\cdot)$. 
Without loss of generality, we let $n=d=j=1$ and
for $\phi =h, f, b, \sigma^i, i=1,..., d $ and $\omega = x, x_\delta, \tilde{x},y, z, xx, xx_\delta, x\tilde{x}, xy, xz, x_\delta x_\delta ,x_\delta \tilde{x}, x_\delta y, x_\delta z, \tilde{x}\tilde{x}, \tilde{x}y, \tilde{x}z, yy, yz,zz$,  we denote
\begin{equation}
	\begin{aligned}\label{notation}
		\phi(t) &\triangleq \phi(t, x^*(t), x^*_\delta(t), \tilde{x}^*(t), y^*(t), z^*(t), u^*(t), \mu^*(t)), \\
		\phi_{\omega}(t) &\triangleq \phi_{\omega}(t, x^*(t), x^*_\delta(t), \tilde{x}^*(t), y^*(t), z^*(t), u^*(t), \mu^*(t)),\\
		\Delta \phi (t) &\triangleq \phi(t, x^*(t), x^*_\delta(t), \tilde{x}^*(t), y^*(t), z^*(t),u^\e(t), \mu^\e(t)) - \phi(t) \\
		& = [\phi(t, x^*(t), x^*_\delta(t), \tilde{x}^*(t), y^*(t), z^*(t),u(t), \mu^*(t)) - \phi(t)]1_{[t_0, t_0+\e]}(t) \\
		&\q+ [\phi(t, x^*(t), x^*_\delta(t), \tilde{x}^*(t), y^*(t), z^*(t),u^*(t), \mu(t)) - \phi(t)]1_{[t_0+\delta, t_0+\delta+\e]}(t), \\
		\Delta \phi_\omega (t)& \triangleq \phi_\omega(t, x^*(t), x^*_\delta(t), \tilde{x}^*(t), y^*(t), z^*(t), u^\e(t), \mu^\e(t)) - \phi_\omega(t)\\
		& = [\phi_\omega(t, x^*(t), x^*_\delta(t), \tilde{x}^*(t), y^*(t), z^*(t),u(t), \mu^*(t)) - \phi_\omega(t)]1_{[t_0, t_0+\e]}(t) \\
		&\q+ [\phi_\omega(t, x^*(t), x^*_\delta(t), \tilde{x}^*(t), y^*(t), z^*(t),u^*(t), \mu(t)) - \phi_\omega(t)]1_{[t_0+\delta, t_0+\delta+\e]}(t),
\end{aligned}\end{equation}
and denote by $\partial^2b$ and $ \partial^2 \sigma$ the Hessian matrices of $b$ and $\sigma$ with respect to $(x, x_{\delta}, \tilde{x})$, respectively.

Consider the following first-order and second-order variational equations for the SDDE in \eqref{system}:
\begin{equation}\label{3.1}\left\{\begin{aligned}
		dx_1(t) =& \Big[b_{x}(t)x_1(t) + b_{x_{\delta}}(t)x_{\delta,1}(t)+ b_{\tilde{x}}(t) \tilde{x}_1 (t) + \Delta b(t)\Big] dt \\
		&+ \Big[\sigma_{x}(t)x_1(t) + \sigma_{x_{\delta}}(t)x_{\delta,1}(t)+ \sigma_{\tilde{x}}(t) \tilde{x}_1 (t) + \Delta \sigma(t)\Big] dW(t),  \quad t\in [0,T]; \\
		x_1(t) =0&, \q t\in [-\delta, 0]
	\end{aligned}\right.\end{equation}
and
\begin{equation}\label{3..1}\left\{\begin{aligned}
		dx_2(t) =& \Big[b_{x}(t)x_2(t) + b_{x_{\delta}}(t)x_{\delta,2}(t) + b_{\tilde{x}}(t)\tilde{x}_2(t) \\
		&+ 
		\frac{1}{2}\big(x_1(t), x_{\delta,1}(t), \tilde{x}_1(t)\big) \partial^2 b(t) \big(x_1(t), x_{\delta,1}(t), \tilde{x}_1(t)\big)^\top\Big]dt \\
		&+ \Big[\sigma_{x}(t)x_2(t) + \sigma_{x_{\delta}}(t)x_{\delta,2}(t)+ \sigma_{\tilde{x}}(t) \tilde{x}_2 (t)  \\
		&+ \frac{1}{2}\big(x_1(t), x_{\delta,1}(t), \tilde{x}_1(t)\big) \partial^2 \sigma(t) \big(x_1(t), x_{\delta,1}(t), \tilde{x}_1(t)\big)^\top\\
		&+\Delta\sigma_x(t)x_1(t) + \Delta\sigma_{x_\delta}(t)x_{\delta,1}(t) + \Delta \sigma_{\tilde{x}}(t)\tilde{x}_1(t)\Big] dW_t,  \quad t\in [0,T]; \\
		x_2(t) =0&, \q t\in [-\delta, 0],
	\end{aligned}\right.\end{equation}
where $x_{\delta, i}(t) \triangleq x_i(t-\delta), \tilde{x}_i(t) \triangleq \int_{-\delta}^{0} e^{\kappa \theta}x_i(t+\theta)d\theta, i=1,2$.
Set $\hat{x}^\e(s) := x^\e(s) - x^*(s), \hat{x}^\e_1(s):= \hat{x}^\e(s) - x_1(s),\hat{x}^\e_2(s):= \hat{x}_1^\e(s) - x_2(s), \big(\hat{y}^\e(t), \hat{z}^\e(t)\big):= \big( y^\e(t) - y^*(t), z^\e(t) - z^*(t)\big)$
 By Meng, Shi, Wang, and Zhang \cite[Lemma 3.1]{meng2025general} and a similar argument of  Hu, Ji, and Xu \cite[Lemma 3.6]{hu2022global}, we have the following result.
\begin{proposition}\label{state}
Let \autoref{A1} hold. Then, for $\beta > 1$, \eqref{3.1} (resp.,\eqref{3..1} )  has a unique solution $x_1 \in S^{\beta}_{\dbF}(-\delta,T;\dbR)$ (resp., $x_2 \in S^{\beta}_{\dbF}(-\delta,T;\dbR)$). Moreover,  
\begin{align*}
&\dbE\Big[ \mathop{\sup}\limits_{s\in [-\delta, T]} |\hat{x}^\e(s)|^{\beta}\Big] = O(\e^\frac{\beta}{2}), \q\q
\dbE\Big[\mathop{\sup}\limits_{s\in [-\delta, T]} |x_1(s)|^\beta\Big] = O(\e^{\frac{\beta}{2}}), \q\q \dbE\Big[\mathop{\sup}\limits_{s\in [-\delta, T]} |\hat{x}_1^\e(s)|^\beta\Big] = O(\e^{\beta}), \\
%
&
\dbE\Big[\mathop{\sup}\limits_{s\in [-\delta, T]} |x_2(s)|^\beta\Big] = O(\e^{\beta}),
\dbE\Big[\mathop{\sup}\limits_{s\in [-\delta, T]} |\hat{x}_2^\e(t)|^\beta\Big] = o(\e^{\beta}),\\
&
\dbE\Big[ \mathop{\sup}\limits_{t\in [0, T]} |\hat{y}^\e(t)|^{\beta} + \Big(\int_{0}^{T}|\hat{z}^\e(t)|^2dt\Big)^{\frac{\beta}{2}}\Big] = O(\e^\frac{\beta}{2}).
%
\end{align*}
\end{proposition}
%


%

%

\section{The decomposition of $\big(\hat{y}^\e, \hat{z}^\e\big)$}\label{dec}
In this section,  we give two decompositions of $\big(\hat{y}^\e, \hat{z}^\e\big)$. 
For notational convenience, we set
\begin{align}
&X^\e(s):=\big(x^\e(s), x^\e_\delta(s), \tilde{x}^\e(s)\big), \q X^*(s):=\big(x^*(s), x^*_\delta(s), \tilde{x}^*(s)\big), \q \hat{X}^\e(s):= X^\e(s)-X^*(s),\\
&
\hat{x}^\e_\delta(s):=  x^\e_\delta(s)-x^*_\delta(s), \q \hat{\tilde{x}}^\e(s):=  \tilde{x}^\e(s)-\tilde{x}^*(s), \q \hat{x}^\e_{\delta, 1}(s):= \hat{x}^\e_{\delta} (s)- x_{\delta, 1}(s), \q \hat{\tilde{x}}_1^\e(s):=  \hat{\tilde{x}}^\e(s)-\tilde{x}_1(s).
\end{align}
By the definition of $(\hat{y}^\e(t), \hat{z}^\e(t))$,
we obtain that
\begin{equation}
\begin{aligned}
&\hat{y}^\e(t) = h\big(X^\e(T)\big) - h(T)  + \int_{t}^{T} \Big\{U_1(s)+ U_2(s)
   \Big\}ds - \int_{t}^{T} \hat{z}^\e(s)dW(s),
\end{aligned}\end{equation}
where
\begin{align*}
&U_1(s):= f\big(s, x^\e(s), x^\e_{\delta}(s), \tilde{x}^\e(s), y^\e(s), z^\e(s), u^\e(s), \mu^\e(s)\big) \\
&\q\q\q\q
- f\big(s, x^*(s)+x_1(s), x_\delta^*(s)+x_{\delta, 1}(s), \tilde{x}^*(s)+\tilde{x}_1(s), y^\e(s), z^\e(s), u^*(s), \mu^*(s)\big),\\
&U_2(s):=f\big(s, x^*(s)+x_1(s), x_\delta^*(s)+x_{\delta, 1}(s), \tilde{x}^*(s)+\tilde{x}_1(s), y^\e(s), z^\e(s), u^*(s), \mu^*(s)\big) - f(s).
\end{align*}
 Applying the first-order Taylor's expansion to $h, U_2$, we obtain that
\begin{equation}\label{haty1}
	\begin{aligned}
		\hat{y}^\e(t) = &\Pi_1^\e(T)+ h_x(T)x_1(T) +h_{x_\delta}(T)x_{\delta, 1}(T) + h_{\tilde{x}}(T)\tilde{x}_1(T)  
		+ \int_{t}^{T} \Big\{f_y(s)\hat{y}^\e(s)  +f_z(s)\hat{z}^\e(s)  + f_x(s)x_1(s) \\
		&+f_{x_\delta}(s)x_{\delta, 1}(s)+ f_{\tilde{x}}(s)\tilde{x}_1(s)
		+ E_1(s)
		+ U_1(s)
		\Big\}ds- \int_{t}^{T} \hat{z}^\e(s) dW(s),
\end{aligned}\end{equation}
where
\begin{align}
&\Pi^\e_1(T):= h\big(X^\e(T)\big)- h(T) - h_x(T)x_1(T) -h_{x_\delta}(T)x_{\delta, 1}(T) - h_{\tilde{x}}(T)\tilde{x}_1(T) \nonumber\\
&\q\q\q
= \int_{0}^{1}\Big\{h_x\big(X^*(T)+\theta \hat{X}^\e(T)\big)- h_x(T)\Big\}d\theta \hat{x}^\e(T) + h_x(T)\hat{x}_1^\e(T)\nonumber\\
&\q\q\q\q
+ \int_{0}^{1}\Big\{h_{x_\delta}\big(X^*(T)+\theta \hat{X}^\e(T)\big)- h_{x_\delta}(T)\Big\}d\theta \hat{x}_\delta^\e(T) + h_{x_\delta}(T)\hat{x}_{\delta, 1}^\e(T)\nonumber\\
&\q\q\q\q+\int_{0}^{1}\Big\{h_{\tilde{x}}\big(X^*(T)+\theta \hat{X}^\e(T)\big)- h_{\tilde{x}}(T)\Big\}d\theta \hat{\tilde{x}}^\e(T) + h_{\tilde{x}}(T)\hat{\tilde{x}}_1^\e(T)
\label{pi}\\
&E_1(s):= \big(f^\e_y(s)-f_y(s)\big) 
\hat{y}^\e(s)
+ \big(f^\e_z(s)-f_z(s)\big)\hat{z}^\e(s)
+ \big(f^\e_x(s)-f_x(s)\big)x_1(s)+ \big(f^\e_{x_\delta}(s)-f_{x_\delta}(s)\big)x_{\delta, 1}(s) \nonumber\\
&
\q\q\q\q+ \big(f^\e_{\tilde{x}}(s)-f_{\tilde{x}}(s)\big)\tilde{x}_1(s)\\
&f_y^{\e}(s):= \int_{0}^{1}f_y\big(s, x^*(s)+\theta x_1(s), x_\delta^*(s)+\theta x_{\delta, 1}(s), \tilde{x}^*(s)+\theta \tilde{x}_1(s), y^*(s)+\theta \hat{y}^\e(s), z^*(s)+\theta \hat{z}^\e(s),\nonumber\\
&\q\q\q\q\q\q\q
 u^*(s), \mu^*(s) \big)d\theta,
\end{align}
and $f^{\e}_x(s), f^{\e}_{x_\delta}(s), f^\e_{\tilde{x}}(s), f^{\e}_{z}(s)$ are defined similarly.\\

{\bf A rough decomposition: }
we now introduce an auxiliary equation to decompose $\big(\hat{y}^\e, \hat{z}^\e\big)$.
\begin{align}\label{ee}
		e(t) = &h_x(T)x_1(T) +h_{x_\delta}(T)x_{\delta, 1}(T) + h_{\tilde{x}}(T)\tilde{x}_1(T)  + \int_{t}^{T}\Big\{f_y(s)e(s)+ f_z(s)r(s)+ f_x(s)x_1(s)\nonumber\\
		&+f_{x_\delta}(s)x_{\delta, 1}(s)+ f_{\tilde{x}}(s)\tilde{x}_1(s)\Big\}ds- \int_{t}^{T} r(s)dW(s), \q t\in [0, T].
	\end{align}
By Briand--Delyon--Hu--Pardoux--Stoica \cite[Theorem 4.2]{briand2003lp} and \autoref{state}, one can obtain the following result:
\begin{lemma}\label{ro est}
Let \autoref{A1} hold. Then, we have that for $\beta>1$,
\begin{align}
\dbE\Big[\mathop{\sup}_{s\in [0, T]}|e(s)|^\beta + \Big(\int_{0}^{T}|r(s)|^2ds\Big)^\frac{\beta}{2}\Big] = O(\e^\frac{\beta}{2}).
\end{align}
\end{lemma}

Define
\begin{align}\label{decom}
\hat{y}^\e_1(s):= \hat{y}^\e(s) - e(s), \q \hat{z}^\e_1(s):= \hat{z}^\e(s) - r(s).
\end{align}
 We now give the estimate of $\big(\hat{y}^\e_1(s), \hat{z}_1^\e(s)\big)$.
\begin{lemma}\label{z1}
Let \autoref{A1} hold. Then,  we have that
\begin{align}\label{aa}
\dbE\Big[\mathop{\sup}_{s\in [0, T]}|\hat{y}^\e_1(s)|^2 + \int_{0}^{T}|\hat{z}^\e_1(s)|^2ds\Big] = O(\e^2).
\end{align}
\end{lemma}
\begin{proof}
By the definition of $\big(\hat{y}_1, \hat{z}_1\big)$, one has that
\begin{equation}
	\begin{aligned}
		\hat{y}_1^\e(t) = \Pi_1^\e(T) + \int_{t}^{T} &\Big\{f_y(s)\hat{y}_1^\e(s)+f_z(s)\hat{z}_1^\e(s) + E_1(s) 
		+  U_1(s)
		\Big\}ds- \int_{t}^{T} \hat{z}_1^\e(s) dW(s).
\end{aligned}\end{equation}
It then follows that
\begin{align}\label{aim1}
	\dbE\Big[\mathop{\sup}_{s\in [0, T]}|\hat{y}^\e_1(s)|^2 + \int_{0}^{T}|\hat{z}^\e_1(s)|^2ds\Big] \leq K\dbE\Big[|\Pi^\e_1(T)|^2 + \Big(\int_{0}^{T}\Big\{ U_1(s)+ E_1(s) \Big\}ds\Big)^2\Big].
\end{align}
Now we estimate the right side of \eqref{aim1}. 
On one hand, we estimate $\dbE[(\int_{0}^{T}|E_1(s)|ds)^2]$ and $\dbE[|\Pi_1^\e(T)|^2]$. By \autoref{A1},  and \autoref{state}, we deduce that
\begin{align*}
&\dbE\Big[\Big(\int_{0}^{T}|f^{\e}_z(s) - f_z(s)| \cd |\hat{z}^\e(s)|ds\Big)^2\Big] \\
&
\leq K \dbE\Big[\Big(\int_{0}^{T}\Big\{|x_1(s)| + |x_{\delta, 1}(s)|  +|\tilde{x}_1(s)| +  |\hat{y}^\e(s)|+|\hat{z}^\e(s)|\Big\} \cd |\hat{z}^\e(s)|ds\Big)^2\Big]= O(\e^2).
\end{align*}
The other terms of $\dbE[(\int_{0}^{T}|E_1(s)|ds)^2]$ can be dedued similarly. In addition, in view of \eqref{pi}, we deduce that
\begin{align}
\dbE\Big[|\Pi^\e_1(T)|^2\Big] 
\leq K \dbE\Big[|\hat{x}^\e(T)|^4+|\hat{x}^\e_\delta(T)|^4+|\hat{\tilde{x}}^\e(T)|^4 + |\hat{x}^\e_1(T)|^2+|\hat{x}^\e_{\delta, 1}(T)|^2+|\hat{\tilde{x}}^\e_1(T)|^2\Big]  = O(\e^2).
\end{align}
On the other hand, we estimate the term $\dbE[(\int_{0}^{T}|U_1(s)|ds)^2]$.
By \autoref{A1} and \autoref{state}, we deduce that
\begin{align*}
&\dbE\Big[\Big(\int_{0}^{T}|U_1(s)|ds\Big)^2\Big]\\
& \leq K\dbE\Big[\Big(\int_{0}^{T} |U_1(s)|1_{[0, T]\backslash\{[t_0, t_0+\e]\cup [t_0+\delta, t_0+\delta+\e]\}}(s)\Big)^2 + \Big(\int_{t_0}^{t_0+\e}|U_1(s)|ds\Big)^2 + \Big(\int_{t_0+\delta}^{t_0+\delta+\e}|U_1(s)|ds\Big)^2\Big] \\
&
\leq K\dbE\Big[\mathop{\sup}_{s\in[0, T]}\Big\{|\hat{x}^\e_1(s)|^2+|\hat{x}^\e_{\delta,1}(s)|^2+|\hat{\tilde{x}}^\e_1(s)|^2\Big\} + \Big(\int_{t_0}^{t_0+\e}\Big\{1+|\hat{x}^\e(s)| + |x^*(s)| + |x_1(s)|+ |\hat{x}_\delta^\e(s)| \\
&\q\q + |x_\delta^*(s)| +|x_{\delta, 1}(s)|+|\hat{\tilde{x}}^\e(s)| + |\tilde{x}^*(s)| + |\tilde{x}_1(s)|+ |\hat{y}^\e(s)|+|y^*(s)|+ |\hat{z}^\e(s)|
+|z^*(s)| + |u(s)|+|u^*(s)|\\
&\q\q
+|\mu^*(s)|  \Big\}ds \Big)^2 + \Big(\int_{t_0+\delta}^{t_0+\delta+\e}\Big\{1+|\hat{x}^\e(s)| + |x^*(s)|+|x_1(s)|+ |\hat{x}_\delta^\e(s)| + |x_\delta^*(s)| +|x_{\delta, 1}(s)|+|\hat{\tilde{x}}^\e(s)| \\
&\q\q+ |\tilde{x}^*(s)|+|\tilde{x}_1(s)|
+ |\hat{y}^\e(s)|+|y^*(s)|+ |\hat{z}^\e(s)|
+|z^*(s)| + |u^*(s)|+|\mu^*(s)|
+|\mu(s)|  \Big\}ds\Big)^2\Big]\leq K_{t_0}\e^2.
\end{align*}
Therefore, we obtain that 
\begin{align*}
\dbE\Big[\Big(\int_{0}^{T}|U_1(s)|ds\Big)^2\Big] = O(\e^2).
\end{align*}

This completes the proof.

\end{proof}
{\bf A precise decomposition:}
we now need to identify the spike variation term in $r$. The analysis will be divided into three steps. 

{\bf Step 1:} motivated by Meng, Shi, Wang, and Zhang \cite{meng2025general}, we give some transformations and use some BSVIEs to eliminate the terms $x_1, x_{\delta, 1}, \tilde{x}_1$.
Define
\begin{equation}\label{x1}
		\bar{H} := \begin{bmatrix}
		h_x(T) & 
		h_{x_\delta}(T)&
		h_{\tilde{x}}(T)
	\end{bmatrix},\q
	\bar{F}(s):=
	\begin{bmatrix}
		f_x(s)& f_{x_\delta}(s)&f_{\tilde{x}}(s)
	\end{bmatrix},
\end{equation}
and 
\begin{equation}\label{x11}
	X_1(t) \triangleq 
	\begin{bmatrix}
		&x_1(t) \\
		& x_{\delta, 1}(t)1_{(\delta, \i)}(t)\\
		& \tilde{x}_1(t)
	\end{bmatrix},\q\q
\end{equation}
and 
\begin{equation*}
	A(t, s) \triangleq 
	\begin{bmatrix}
		b_x(s) & b_{x_\delta}(s) & b_{\tilde{x}}(s)\\
		1_{(\delta, \i)}(t-s)b_x(s) & 1_{(\delta, \i)}(t-s)b_{x_\delta}(s) & 1_{(\delta, \i)}(t-s)b_{\tilde{x}}(s)\\
		I_{n\times n} & -e^{-\kappa \delta}I_{n\times n} & -\kappa I_{n\times n}
	\end{bmatrix},
\end{equation*}
\begin{equation*}
	C(t, s) \triangleq 
	\begin{bmatrix}
		\sigma_x(s) & \sigma_{x_\delta}(s) & \sigma_{\tilde{x}}(s)\\
		1_{(\delta, \i)}(t-s)\sigma_x(s) & 1_{(\delta, \i)}(t-s)\sigma_{x_\delta}(s) & 1_{(\delta, \i)}(t-s)\sigma_{\tilde{x}}(s)\\
		0 & 0 & 0
	\end{bmatrix},
\end{equation*}
\begin{equation*}
	B(t, s) \triangleq 
	\begin{bmatrix}
		&\Delta b(s)\\
		&1_{(\delta, \i)}(t-s)\Delta b(s)\\
		&0
	\end{bmatrix},\q\q
	D(t, s) \triangleq 
	\begin{bmatrix}
		&\Delta \sigma(s)\\
		&1_{(\delta, \i)}(t-s)\Delta \sigma(s)\\
		&0
	\end{bmatrix},
\end{equation*}
It is obvious that
\begin{align}
A(\cd, s), B(\cd, s), C(\cd, s), D(\cd, s) \in \sF_s.
\end{align}
In view of \eqref{3.1} and \eqref{3..1}, we have
\begin{align}
	&X_1 (t) = \int_{0}^{t}[A(t,s)X_1(s) + B(t,s)]ds + \int_{0}^{t}[C(t,s)X_1(s) + D(t,s)]dW(s),\q t\in [0, T].\label{X11}
\end{align}
By Wang and Yong \cite[page 3613]{wang2023spike} and \autoref{A1}, we obtain that \eqref{X11} admit unique solutions $X_1 \in C_{\dbF}\big([0, T]; L^\beta(\Om; \dbR^{3})\big)$ for $\beta > 1$.
Set
\begin{align}\label{g}
\Gamma^t(s) := \exp\Big\{\int_{t}^{s}f_y(\theta)d\theta\Big\}\mathcal{E}\Big(\int_{t}^{s}f_z(\theta)dW(\theta)\Big).
\end{align}

\begin{proposition}\label{ga}
Let \autoref{A1} hold. Then, we have that for $0\leq t \leq T$ and $\beta>1$,
\begin{align}
\dbE\big[\mathop{\sup}_{t\in [0, T]}\mathop{\sup}_{s\in [t, T]}|\Gamma^t(s)|^\beta\big]<\i, \q \dbE\big[\mathop{\sup}_{t\in [0, T]}\mathop{\sup}_{s\in [t, T]}|\Gamma^t(s)^{-1}|^\beta\big]<\i
\end{align}
\end{proposition}
\begin{proof}
	Since 
	\begin{align}
		&\Gamma^t(s) = \exp\Big\{\int_{t}^{s}f_y(\theta)d\theta + \int_{0}^{s}f_z(\theta)dW(\theta) - \int_{0}^{t}f_z(\theta)dW(\theta) - \frac{1}{2}\int_{t}^{s} |f_z(\theta)|^2d\theta\Big\}\nonumber\\
		& \leq \exp\Big\{\|f_y\|_\i T + 2\mathop{\sup}_{k\in[0, T]}\Big|\int_{0}^{k}f_z(\theta)dW(\theta)\Big|\Big\}.\label{gg}
	\end{align}
In addition, for $0\leq t\leq T$ and $\beta>1$, it is easy to verify that $\dbE\Big[\exp\Big\{\frac{1}{2}\int_{0}^{T}\beta^2|f_z(s)|^2 ds\Big\} \Big]< \i$, which implies that $\Big\{\mathcal{E}\Big(\int_{0}^{t}\beta f_z (s)dW(s)\Big)\Big\}_{t\in [0, T]}$ is a martingale. It then follows from \autoref{A1} and Doob's inequality that
\begin{align}\label{G1}
	\dbE \Big[\mathop{\sup}\limits_{k \in [0,T]} \Big\{\int_{0}^{k}f_{z}(s)dW_s\Big\}^\beta\Big]\leq K\dbE \Big[ \Big\{\int_{0}^{T}f_{z}(s)dW_s\Big\}^\beta\Big]
	\leq K \dbE \Big[ \mathcal{E}\Big(\int_{t}^{T}\beta f_{z}(s)dW(s)\Big) \Big]\leq K<\i.
	%
\end{align}
	Combining  \eqref{gg} and \eqref{G1}, we obtain that for $\beta>1$,
\begin{align}\label{g0}
	\dbE[\mathop{\sup}_{t\in[0, T]}\mathop{\sup}_{s\in[t, T]}|\Gamma^t(s)|^\beta]<\i.
\end{align}
Thus, we obtain that for $\beta > 1$,
Similarly, one can obtain that $\dbE\big[\mathop{\sup}_{t\in [0, T]}\mathop{\sup}_{s\in [t, T]}|\Gamma^t(s)^{-1}|^\beta\big]<\i$.
\end{proof}

Applying It$\hat{\rm o}$'s formula to $\big(\Gamma^t(s)e(s)\big)_{s\in[t, T]}$, one can obtain that
\begin{align*}
e(t) = \Gamma^t(T)\bar{H}X_1(T)+\int_{t}^{T}\Gamma^t(s)\bar{F}(s)X_1(s)ds - \int_{t}^{T}\Gamma^t(s)[r(s)+f_z(s)e(s)]dW(s).
\end{align*}

{\bf Step 2:} in this step, we will introduce some adjoint equations to eliminate the terms $x_1, x_{\delta, 1}, \tilde{x}_1$.

The following is a family of adjoint equations parameterized by $t$:
\begin{equation}\label{adjointe}\left\{\begin{aligned}
		%
		&(a)\q \lambda^t(s) = \big(\Gamma^t(T)\bar{H}\big)^\top - \int_{s}^{T}\nu^t(\theta)dW(\theta), \q s\in [t, T];\\
		&(b)\q \eta^t(s) = \big(\Gamma^t(s)\bar{F}(s)\big)^{\top} + A(T, s)^\top\lambda^t(s)+ C(T,s)^\top \nu^t(s) 
		\\
		&\q\q\q\q\q+ \int_{s}^{T}\Big[A(\theta,s)^\top \eta^t(\theta)+ C(\theta,s)^\top \zeta^t(\theta,s)\Big]d\theta - \int_{s}^{T}\zeta^t(s,\theta)dW(\theta), \q s\in [t, T];\\
		&(c) \q \eta^t(s) = \dbE_t[\eta^t(s)] +  \int_{t}^{s} \zeta^t(s,\theta)dW(\theta) ,\q s\in [t, T].
	\end{aligned}\right.\end{equation}

By \autoref{ga}, Briand, Delyon, Hu, Pardoux, and Stoica \cite[Theorem 4.2]{briand2003lp}, and  Hamaguchi and Taguchi \cite[Theorem 3.4]{hamaguchi2023approximations}, we have the following result.
\begin{proposition}\label{aua}
Let \autoref{A1} hold. Then, for $0\leq t \leq T$ and $\beta>1$, 
\eqref{adjointe} admits a unique solution satisfying $(\lambda^t(\cd), \nu^t(\cd)) \in S^{\beta}_{\dbF}(t, T;\dbR^3) \times \mathcal{M}^{2, \beta}_{\dbF}(t, T;\dbR^{3})$,
$\eta^t(s), \zeta^t(\theta, s) \in \sF_s$, 
and
\begin{align}
\dbE\Big[\Big(\int_{t}^{T}|\eta^t(s)|^2ds\Big)^\beta\Big] + \dbE\Big[\Big(\int_{t}^{T}\int_{t}^{T}|\zeta^t(\theta, s)|^2dsd\theta\Big)^\beta\Big]<\i.
\end{align}
\end{proposition}

Since $X_1$ is a SVIE, we cannot apply It$\hat{\rm o}$'s formula directly. To handle this, we
define
\begin{align}\label{xx}
	\mathcal{X}_1(s, \theta)=\int_{0}^{\theta} [A(s,k)X_1(k) + B(s,k)]dk + \int_{0}^{ \theta}[C(s,k)X_1(k) + D(s,k)]dW(k),\q 0\leq  \theta \leq s\leq T.
\end{align}
One can check that 
\begin{align}\label{pro}
	\mathcal{X}_1(s, s) = X_1(s), \q 0\leq s\leq T,  \q \mathop{\sup}_{0\leq s\leq T}\dbE[\mathop{\sup}_{0\leq k\leq s}|\mathcal{X}_1(s, k)|^\beta] = O(\e^\frac{\beta}{2}), \q \hb{and} \q \mathcal{X}_1(\cd, s) \in \sF_s.
\end{align}
First, applying It$\hat{\rm o}$'s formula to $\big(\lambda^t(\theta)\mathcal{X}_1(T,\theta)\big)_{\theta \in [t, T]}$, we obtain that
\begin{align}
e(t) =& \lambda^t(t)^\top\mathcal{X}_1(T, t) + \int_{t}^{T}\Big\{\big[\big(\Gamma^t(s)\bar{F}(s)\big)^\top + A(T, s)^\top \lambda^t(s) + C(T, s)^\top\nu^t(s)\big]X_1(s)\Big\}ds + \int_{t}^{T}\Big\{B(T,s)^\top\nonumber\\
&
\lambda^t(s)+D(T, s)^\top\nu^t(s)\Big\}ds - \int_{t}^{T}\Big\{-\Big[\lambda^t(s)^\top\big\{C(T, s) X_1(s)+D(T, s)\big\}+\nu^t(s)^\top\mathcal{X}_1(T, s)\Big]\nonumber\\
&
+\Gamma^t(s)[r(s)+f_z(s)e(s)]\Big\}dW(s).
\end{align}
Second, define
\begin{align}\label{delta}
\Delta^t(s, \theta) := \eta^t(s) - \int_{\theta}^{s}\zeta^t(s, k)dW(k), \q t\leq \theta \leq s.
\end{align}
we now prove that $\Delta^t(s, \theta) \in \sF_\theta$. By \eqref{adjointe}(c) and the fact that $\zeta^t(s, k) \in \sF_k$, one can obtain that
\begin{align}
&\dbE_\theta[\eta^t(s)] = \dbE_t[\eta^t(s)] + \dbE_\theta\Big[\int_{t}^{s}\zeta^t(s, k)dW(k)\Big]  = \dbE_t[\eta^t(s)] + \int_{t}^{\theta}\zeta^t(s, k)dW(k)\nonumber \\
&
= \eta^t(s) - \int_{\theta}^{s}\zeta^t(s, k)dW(k) = \Delta^t(s, \theta).\label{du}
\end{align}
Thus, we have prove that $\Delta^t(s, \theta) \in \sF_\theta$. Applying It$\hat{\rm o}$
's formula to $\big(\Delta^t(s, \theta)\mathcal{X}_1(s, \theta)\big)_{\theta\in [t, s]}$ note the fact that
$\int_{0}^{T}\int_{0}^{s}a(s, t)dtds = \int_{0}^{T}\int_{t}^{T}a(s, t)dsdt = \int_{0}^{T}\int_{s}^{T}a(t, s)dtds$ and Da Prato and Zabczyk \cite[Theorem 4.33]{da2014stochastic}, one can deduce that
\begin{align}\label{1}
	&e(t) = \lambda^t(t)^\top\mathcal{X}_1(T, t) + \int_{t}^{T}\Delta^t(s, t)^\top\mathcal{X}_1(s, t)ds + \int_{t}^{T}\Big\{\int_{s}^{T}\big[B(\theta, s)^\top\Delta^t( \theta, s)  + D( \theta, s)^\top\zeta^t(\theta, s) \big]d\theta \nonumber\\
	&+B(T,s)^\top\lambda^t(s) +D(T, s)^\top\nu^t(s) \Big\}ds 
	 - \int_{t}^{T}\Big\{\Gamma^t(s)[r(s)+f_z(s)e(s)] -r^t_1(s)\Big\}dW(s) \nonumber\\
	 &+ \int_{t}^{T}J_1(s)ds-\int_{t}^{T}J_2(s)ds,
\end{align}
where
\begin{align}
&r^t_1(s):=\lambda^t(s)^\top\big\{C(T, s)X_1(s)+D(T, s)\big\}+\nu^t(s)^\top\mathcal{X}_1(T, s) + \int_{s}^{T}\Big\{\Delta^t(\theta, s)^\top \Big[C(\theta, s) X_1(s)+D( \theta, s)\Big]\nonumber\\ 
&\q\q\q+\zeta^t(\theta, s)^\top\mathcal{X}_1( \theta, s)\Big\}d\theta, \\\label{r1}
&J_1(s):=\int_{s}^{T} \zeta^t(s, \theta)^\top X_1(s) dW(\theta), \q J_2(s):=\int_{t}^{s}\int_{\theta}^{s}\zeta^t(s, \theta')^\top A(s, \theta) X_1(\theta) dW(\theta') d\theta.
\end{align}
Taking the conditional expectation $\dbE_s[\cd]$ to \eqref{adjointe}(b) and  by the fact that
\begin{align*}
\eta^t(\theta) = \dbE_s[\eta^t(\theta)] +  \int_{s}^{\theta} \zeta^t(\theta,\theta')dW(\theta'),
\end{align*}
 one can obtain that
\begin{align*}
&\eta^t(s) = \big(\Gamma^t(s)\bar{F}(s)\big)^{\top} + A(T, s)^\top\lambda^t(s)+ C(T,s)^\top \nu^t(s) 
\\
&\q\q\q\q\q+ \int_{s}^{T}\Big[A(\theta,s)^\top \eta^t(\theta)+ C(\theta,s)^\top \zeta^t(\theta,s)\Big]d\theta - \int_{s}^{T}A(\theta, s)^\top\int_{s}^{\theta}\zeta^t(\theta,\theta')dW(\theta')d\theta.
\end{align*}
In view of \eqref{adjointe}(b), one can deduce that
\begin{align*}
\int_{s}^{T}A(\theta, s)^\top\int_{s}^{\theta}\zeta^t(\theta,\theta')dW(\theta')d\theta = \int_{s}^{T}\zeta^t(s,\theta)dW(\theta).
\end{align*}
By Da Prato and Zabczyk \cite[Theorem 4.33]{da2014stochastic}, we deduce that
\begin{align}\label{im}
	\zeta^t(s, \theta)^\top = \int_{\theta}^{T}\zeta^t(\theta', \theta)^\top A(\theta', s) d\theta'.
\end{align}
It then follows from \eqref{im} and Da Prato and Zabczyk \cite[Theorem 4.33]{da2014stochastic} that
\begin{align}\label{12}
&\int_{t}^{T}J_2(s)ds = \int_{t}^{T}\int_{t}^{s}\int_{\theta}^{s}\zeta^t(s, \theta')^\top A(s, \theta)X_1(\theta) dW(\theta') d\theta ds\nonumber\\
&= \int_{t}^{T}\int_{s}^{T}\int_{s}^{\theta}\zeta^t(\theta, \theta')^\top A(\theta, s) X_1(s)dW(\theta')d\theta ds= \int_{t}^{T}\int_{s}^{T}\int_{\theta}^{T}\zeta^t(\theta', \theta)^\top A(\theta', s) d\theta'X_1(s)dW(\theta) ds\nonumber\\
&=\int_{t}^{T}\int_{s}^{T}\int_{s}^{\theta}\zeta^t(\theta, \theta')^\top A(\theta, s) X_1(s)dW(\theta')d\theta ds= \int_{t}^{T}\int_{s}^{T}\zeta^t(s, \theta)^\top X_1(s)dW(\theta) ds = \int_{t}^{T}J_1(s)ds.
\end{align}
{\bf Step 3:} in this step, we give the precise transformation. 
Let us introduce the following equation:
\begin{align}\label{r2}
	e_2(t) =& \int_{t}^{T}\Big\{\int_{s}^{T}\big[B(\theta, s)^\top\Delta^t( \theta, s)  + D( \theta, s)^\top\zeta^t(\theta, s) \big]d\theta 
	+B(T,s)^\top\lambda^t(s) +D(T, s)^\top\nu^t(s) \Big\}ds \nonumber\\
	&
	- \int_{t}^{T}r^t_2(s) dW(s),
\end{align}
and  let 
\begin{align}\label{y1}
	\tilde{e}(t):=e(t) - \lambda^t(t)^\top\mathcal{X}_1(T, t) - \int_{t}^{T}\Delta^t(\theta, t)^\top\mathcal{X}_1(\theta, t)d\theta - e_2(t).
\end{align}
Combining \eqref{1} and \eqref{12}, we obtain that
\begin{align}\label{34}
\tilde{e}(t) = \int_{t}^{T}\big\{\Gamma^t(s)[r(s)+f_z(s)e(s)]-r^t_1(s) - r^t_2(s)\big\}dW(s).
\end{align}
Taking conditional expectation $\dbE_t[\cd]$ on both sides of \eqref{34}, one can obtain that
\begin{align}\label{35}
\tilde{e}(s) = 0, s\in [t, T]\q \hb{and} \q r(s) = \Gamma^t(s)^{-1}\big(r_1^t(s) + r^t_2(s)\big) -f_z(s)e(s), \q a.e..
\end{align}
Since the terms $\lambda^t(t)$ and $\Delta^t(\theta, t), \nu^t(s), \zeta^t(\theta, s)$ appear in \eqref{1} and \eqref{r1}, we 
need the following result. 
\begin{proposition}\label{lv}
	Let \autoref{A1} hold. Then, we have that for $0\leq t\leq T$,
	\begin{align}\label{a}
		\big(\lambda^t(s), \nu^t(s)\big) = \big(\Gamma^0(t)^{-1}\lambda^0(s), \Gamma^0(t)^{-1}\nu^0(s)\big), \q \big(\eta^t(s), \zeta^t(s, \theta)\big) = \big(\Gamma^0(t)^{-1} \eta^0(s), \Gamma^0(t)^{-1}\zeta^0(s, \theta)\big).
	\end{align}
\end{proposition}
\begin{proof}
	Since 
	\begin{align*}
		\Gamma^0(t)^{-1}\Gamma^0(T) = \Gamma^t(T), \q \Gamma^0(t)^{-1} \Gamma^0(s) = \Gamma^t(s),
	\end{align*}
	one can check that 
	$\big(\Gamma^0(t)^{-1}\lambda^0(s), \Gamma^0(t)^{-1}\nu^0(s), \Gamma^0(t)^{-1} \eta^0(s), \Gamma^0(t)^{-1}\zeta^0(s, \theta)\big)$ is  the solution of \eqref{adjointe}. Then, by the unique property of \eqref{adjointe}, one can obtain that \eqref{a} holds.
\end{proof}
In view of \eqref{y1}, \eqref{35}, and \eqref{34}, we give the precise decomposition of $ \big(\hat{y}^\e(s), \hat{z}^\e(s)\big)$.
\begin{align}
	&\hat{y}^\e(s) = \hat{y}^\e_1(s)  + e(s)= \hat{y}^\e_1(s)+e_1(s) + e_2(s) \label{dey}\\
	&\hat{z}^\e(s)= \hat{z}_1^\e(s)+r(s)= \hat{z}_1^\e(s) +\Gamma^0(s)^{-1}\big\{r_1^0(s) + r^0_2(s)\big\}-f_z(s)e(s) \nonumber\\
	& = \hat{z}_1^\e(s)+ \Gamma^0(s)^{-1}\big\{r^0_{11}(s)+r^0_{12}(s)+r^0_2(s)\big\}-f_z(s)e_2(s), a.e..\label{de}
\end{align}
where
\begin{align}
	&e_1(s):=\Gamma^0(s)^{-1}\Big\{\lambda^0(s)^\top\mathcal{X}_1(T, s) + \int_{s}^{T}\dbE_s[\eta^0(\theta)]^\top\mathcal{X}_1(\theta, s)d\theta\Big\},\\
	&r^0_{11}(s):= \lambda^0(s)^\top C(T, s)X_1(s)+ \{-f_z(s)\lambda^0(s)^\top +\nu^0(s)^\top\}\mathcal{X}_1(T, s) + \int_{s}^{T}\Big\{\dbE_s[\eta^0(\theta)]^\top C(\theta, s) X_1(s)\nonumber\\
	&\q\q\q\q
	+\big\{-f_z(s)\dbE_s[\eta^0(\theta)]^\top +\zeta^0(\theta, s)^\top\big\}\mathcal{X}_1( \theta, s)\Big\}d\theta, \\
	&
	r^0_{12}(s):=\lambda^0(s)^\top D(T, s)+ \int_{s}^{T}\dbE_s[\eta^0(\theta)]^\top D( \theta, s)d\theta.\label{r21}
\end{align}

Next, we give estimates for $r_{11}^0, r_{12}^0, e_2, r_2^0$.
\begin{proposition}\label{e2p}
Let \autoref{A1} hold. Then, we have that 
\begin{align}
&(i)\q\dbE\Big[\int_{t_0}^{t_0+\e}|r_{11}^0(s)|ds + \int_{t_0+\delta}^{t_0+\delta+\e}|r_{11}^0(s)|ds \Big] = o(\e), \q \dbE\Big[\int_{0}^{T}|r^0_{12}(s)|^2ds\Big] = O(\e). \label{r11}\\
&(ii) \q\dbE\Big[\mathop{\sup}_{s\in[0, T]}|e_2(s)|^2 \Big]=o(\e), \q
\dbE\Big[\int_{0}^{T}|r_2^0(s)|^2ds \Big] = o(\e).\label{e2}
\end{align}
\end{proposition}
\begin{proof}
(i)  By H$\ddot{\rm o}$lder's inequality and \autoref{state} that 
\begin{align*}
&\dbE\Big[\int_{t_0}^{t_0+\e}|r_{11}^0(s)|ds  \Big]
\leq K\dbE\Big[\int_{t_0}^{t_0+\e}\Big( |\lambda^0(s)| +|\nu^0(s)| + \int_{s}^{T}\big\{|\eta^0(s) | 
+|\zeta^0(\theta, s)|\big\}d\theta\Big)^2 ds\Big]^\frac{1}{2}\e \leq K\e^{\frac{3}{2}} = o(\e).
\end{align*}
Similarly, we have that
\begin{align*}
\dbE\Big[\int_{t_0+\delta}^{t_0+\delta+\e}|\Gamma^0(s)^{-1}r_{11}^0(s)|ds  \Big] = o(\e).
\end{align*}

Furthermore, by the definition of $D(\cd, \cd)$, H$\ddot{\rm o}$lder's inequality, and \autoref{aua}, one can deduce that
\begin{align*}
\dbE\Big[\int_{0}^{T}|r^0_{12}(s)|^2ds\Big] = O(\e).
\end{align*}

(ii) on one hand, in view of \eqref{r2} and the definition of $B(\cd, \cd), D(\cd, \cd)$, by H$\ddot{\rm o}$lder's inequality and Doob's inequality, one can deduce that there exist $G_{11}(\cd, \cd), G'_{11}(\cd, \cd), G_{12}(\cd), G'_{12}(\cd)$ (independent of $\e$) such that 
\begin{align}
\dbE\Big[\mathcal{\sup}_{\theta\in[0, T]}\mathcal{\sup}_{s\in[0, T]}\{|G_{11}(\theta, s)|^\beta+|G'_{11}(\theta, s)|^\beta\}+\mathcal{\sup}_{s\in[0, T]}\{|G_{12}(s)|^\beta+|G'_{12}(s)|^\beta\}\Big]<\i.
\end{align}
 and 
\begin{align*}
&\dbE\big[\mathop{\sup}_{t\in [0, T]}|e_2(t)|^2\big]\leq K \dbE\Big[\Big\{\mathcal{\sup}_{\theta\in[0, T]}\mathcal{\sup}_{s\in[0, T]}|G_{11}(\theta, s)|+\mathcal{\sup}_{s\in[0, T]}|G_{12}( s)|\Big\}^2\Big|\int_{t_0}^{t_0+\e}\Big(\int_{0}^{T}\big(\dbE_s[|\eta^0(\theta)|]  \\
&+ |\zeta^0(\theta, s)| \big)d\theta 
+|\lambda^0(s)|+|\nu^0(s)|\Big)ds\Big|^2\Big]+K\dbE\Big[\Big\{\mathcal{\sup}_{\theta\in[0, T]}\mathcal{\sup}_{s\in[0, T]}|G'_{11}(\theta, s)|+\mathcal{\sup}_{s\in[0, T]}|G'_{12}( s)|\Big\}^2\\
&
\Big|\int_{t_0+\delta}^{t_0+\delta+\e}\Big(\int_{0}^{T}\big(\dbE_s[|\eta^0(\theta)|]  + |\zeta^0(\theta, s)| \big)d\theta 
+\big(|\lambda^0(s)|+|\nu^0(s)|\Big)ds\Big|^2\Big].
\end{align*}
By  \autoref{aua} , one can deduce that for $\beta>1$,
\begin{align*}
&\dbE\Big[\Big|\int_{t_0}^{t_0+\e}\int_{0}^{T}\zeta^0(\theta, s)d\theta ds\Big|^{2\beta}\Big]^\frac{1}{\beta}=\dbE\Big[\Big|\int_{0}^{T}\int_{t_0}^{t_0+\e}\zeta^0(\theta, s)dsd\theta \Big|^{2\beta}\Big]^\frac{1}{\beta} \leq \dbE\Big[\Big|\int_{0}^{T}\Big(\int_{t_0}^{t_0+\e}|\zeta^0(\theta, s)|^2ds \Big)^\frac{1}{2}d\theta\Big|^{2\beta}\Big]^\frac{1}{\beta}\e\\
&
\leq \dbE\Big[\Big|\int_{t_0}^{t_0+\e}\int_{0}^{T}|\zeta^0(\theta, s)|^2d\theta ds\Big|^{\beta}\Big]\e.
\end{align*}
It then follows from Dominated Convergence theorem that $\dbE[\mathcal{\sup}_{s\in [0, T]}|e_2(s)|^2]=o(\e)$.
On the other hand, let
$g(s):=\int_{s}^{T}\big[B(\theta, s)^\top\Delta^0( \theta, s)  + D( \theta, s)^\top\zeta^0(\theta, s) \big]d\theta 
+B(T,s)^\top\lambda^0(s) +D(T, s)^\top\nu^0(s)$.
Then, in view of \eqref{r2}, one can obtain that
\begin{align*}
\int_{0}^{T}r_2^0(s)dW(s) 
= \int_{0}^{T} g(s)ds - e_2(0).
\end{align*}
It then follows from Burkholder-Davis-Gundy's inequality that for $1<\beta<2$,
\begin{align*}
\dbE\Big[\int_{0}^{T}|r_2^0(s)|^2ds \Big] \leq K \dbE\Big[\Big(\int_{0}^{T}|g(s)|ds\Big)^2+|e_2(0)|^2\Big].
\end{align*}
Similary, we can obtain that $\dbE\Big[\int_{0}^{T}|r_2^0(s)|^2ds \Big] = o(\e).$

This completes the proof.
\end{proof}

\section{Analysis of duality relations for the terms $X_1^\top FX_1$, $eF_{1y}X_1, rF_{1z}X_1 $ }\label{first ad}
In this section, we use the duality method established in Meng, Shi, Wang, and Zhang \cite{meng2025general}, the decomposition established in \autoref{dec}, and a new approach to eliminate the terms $X_1^\top FX_1$, $eF_{1y}X_1, rF_{1z}X_1 $.

Let the variational equation be the following form:
\begin{equation}\label{va}
	\begin{aligned}
		\hat{y}(t) = & \int_{t}^{T} \Big\{f_y(s)\hat{y}(s) + f_z(s)\hat{z}(s) + I(s)\Big\}ds-\int_{t}^{T}\hat{z}(s) dW(s),
\end{aligned}\end{equation}
where $I(s)$ will be determined later. 
Recall that $\big(\hat{y}^\e(s), \hat{z}^\e(s) \big)= \big(y^\e(s)-y^*(s), z^\e(s)-z^*(s)\big)$.
set 
\begin{align}\label{v2}
	\big(\hat{y}_2^\e(s), \hat{z}_2^\e(s)\big) := \big(\hat{y}^\e(s) - \hat{y}(s), \hat{z}^\e(s) - \hat{z}(s)\big).
\end{align}
%
%
%
By the definition of $\big(\hat{y}^\e_2, \hat{z}^\e_2\big)$, we deduce that
\begin{equation}\label{y2-1}
	\begin{aligned}
		\hat{y}^\e_2(t) =& h(X^\e(T)) - h(T)  
		+ \int_{t}^{T} \Big\{V_1(s)
		+ V_2(s) - f_y(s)\hat{y}(s) - f_z(s)\hat{z}(s)- \Big[I(s) -\tilde{\Delta} f(s)\Big]
		\Big\}ds \\
		&- \int_{t}^{T} \hat{z}^\e_2 (s) dW(s),
	\end{aligned}
\end{equation}
where 
\begin{equation}\label{v1}
	\begin{aligned}
		&V_1 (s):= f\big(s, x^\e(s), x^\e_\delta(s),\tilde{x}^\e(s), y^\e(s) , z^\e(s), u^\e(s), \mu^\e(s)\big) - f\big(s, x^*(s)+x_1(s)+x_2(s), \\
		&\q\q\q\q x^*_\delta (s)+x_{\delta,1}(s)+x_{\delta, 2}(s), \tilde{x}^*(s)+\tilde{x}_{1}(s)+\tilde{x}_{ 2}(s), 
		y^\e(s), z^\e(s), u^*(s), \mu^*(s)\big)- \tilde{\Delta} f(s) ,\\
		&V_2(s):= f\big(s, x^*(s)+x_1(s)+x_2(s), x_{\delta}^*(s)+x_{\delta, 1}(s)+x_{\delta, 2}(s), \tilde{x}^*(s)+\tilde{x}_{1}(s)+\tilde{x}_{ 2}(s),  y^\e(s), z^\e(s),\\
		&\q\q\q\q\q
		 u^*(s), \mu^*(s)\big) - f(s),\\
	    &\tilde{\Delta}f(s) := f\big(s, x^*(s), x^*_\delta(s), \tilde{x}^*(s), y^*(s), z^*(s)+\Gamma^0(s)^{-1}r^0_{12}(s), u^\e(s), \mu^\e(s)\big) \\
	    &\q\q\q\q\q
	    - f\big(s, x^*(s), x^*_\delta(s), \tilde{x}^*(s), y^*(s), z^*(s)+\Gamma^0(s)^{-1}r^0_{12}(s), u^*(s), \mu^*(s)\big),
\end{aligned}\end{equation}
and  $ r_{12}^0(s), \Gamma^0(s)$ are defined in \eqref{r21}, \eqref{g}.
Recall that $\bar{H}, \bar{F}$ and $X_1(s)$ are defined in \eqref{x1} and \eqref{x11}, respectively.  Define
\begin{equation}
	H:= 
	\begin{bmatrix}
	h_{xx}(T)&h_{xx_\delta}(T)&h_{x\tilde{x}}(T)\\
	h_{xx_\delta}(T) & h_{x_\delta x_\delta}(T)&h_{x_\delta \tilde{x}}(T)\\
	h_{x\tilde{x}}(T)&h_{x_\delta \tilde{x}}(T)&h_{\tilde{x}\tilde{x}}(T)
	\end{bmatrix},
	F(s):=
	\begin{bmatrix}
	f_{xx}(s)&f_{xx_\delta}(s)&f_{x\tilde{x}}(s)\\
	f_{xx_\delta}(s) & f_{x_\delta x_\delta}(s)&f_{x_\delta \tilde{x}}(s)\\
	f_{x\tilde{x}}(s)&f_{x_\delta \tilde{x}}(s)&f_{\tilde{x}\tilde{x}}(s)
	\end{bmatrix}.
\end{equation}
and 
\begin{equation*}
	X_2(t) :=
	\begin{bmatrix}
		&x_2(t)\\
		&x_{\delta, 2}(t)1_{(\delta, \i)}(t)\\
		& \tilde{x}_2(t)
	\end{bmatrix},\q
	\bar{B}(t, s) :=
	\begin{bmatrix}
		&\frac{1}{2}X_1(s)^{\top}\partial^2b(s)X_1(s)\\
		&\frac{1}{2}1_{(\delta, \i)}(t-s)X_1(s)^{\top}\partial^2b(s)X_1(s)\\
		&0
	\end{bmatrix},\q
\end{equation*}
\begin{equation*}
	\bar{D}(t, s) :=
	\begin{bmatrix}
		&\frac{1}{2}X_1(s)^{\top}\partial^2\sigma(s)X_1(s) + \Delta \Xi(s)X_1(s)\\
		&1_{(\delta, \i)}(t-s)[\frac{1}{2}X_1(s)^{\top}\partial^2\sigma(s)X_1(s) + \Delta \Xi(s)X_1(s)]\\
		&0
	\end{bmatrix},\q
	\Delta \Xi(s) :=
	\begin{bmatrix}
		\Delta \sigma_x(s), \Delta \sigma_{x_\delta}(s), \Delta \sigma_{\tilde{x}}(s)
	\end{bmatrix}.
\end{equation*}
In view of \eqref{3.1} and \eqref{3..1}, we have
\begin{align}
	X_2 (t) = \int_{0}^{t}[A(t,s)X_2(s) + \bar{B}(t,s)]ds +  \int_{0}^{t}[C(t,s)X_2(s) + \bar{D}(t,s)]dW(s),\q t\in [0, T].\label{X2}
\end{align}
By applying the second-order Taylor's expansion to $h$ and $V_2(s)$ and note the fact that
$\big(\hat{y}^\e(s), \hat{z}^\e(s)\big) =  \big(\hat{y}_1^\e(s)+ e(s), \hat{z}_1^\e(s)+ r(s)\big)$, \eqref{y2-1} can be changed to 
\begin{equation}\label{y2-2}
	\begin{aligned}
		\hat{y}^\e_2(t) =&  \Pi^\e_2(T) + \bar{H}\big(X_1(T)+X_2(T)\big)+ \frac{1}{2} X_1(T)^\top H X_1(T)
		+ \int_{t}^{T}  \Big\{f_y(s)\hat{y}_2^\e(s)+f_z(s)\hat{z}_2^\e(s)+V_1(s)  \\
		&+ \bar{F}(s)\big(X_1(s)+X_2(s)\big)+ \frac{1}{2} X_1(s)^\top F(s)X_1(s)
		  + \frac{1}{2}f^\e_{yy}(s)\big(\hat{y}_1^\e(s)+e(s)\big)^2  + \frac{1}{2}f^\e_{zz}(s) \big(\hat{z}_1 ^\e(s) + r(s)\big)^2
		  \\
		 &	+f^\e_{xy}(s)\big(x_1(s)+x_2(s)\big)\big(\hat{y}_1^\e(s)+e(s)\big) + f^\e_{x_\delta y}(s)\big(x_{\delta, 1}(s)+x_{\delta, 2}(s)\big)	\big(\hat{y}_1^\e(s)+e(s)\big) 
		 + f^\e_{\tilde{x} y}(s)\big(\tilde{x}_{ 1}(s) \\
		 &
		 +\tilde{x}_{2}(s)\big)	\big(\hat{y}_1^\e(s)+e(s)\big)+ f^\e_{xz}(s)  \big(x_1(s)+x_2(s)\big)\big(\hat{z}_1^\e(s)+r(s)\big)
		 +f^\e_{x_\delta z}(s)\big(x_{\delta,1}(s)+x_{\delta, 2}(s)\big)\big(\hat{z}_1^\e(s)\\
		&+r(s)\big)
		+ f^\e_{\tilde{x} z}(s)\big(\tilde{x}_{1}(s)+\tilde{x}_{ 2}(s)\big)\big(\hat{z}_1^\e(s)+r(s)\big)+f_{yz}^\e(s)\big(\hat{y}_1^\e(s)+e(s)\big)\big(\hat{z}_1^\e(s)+r(s)\big)- [I(s)  \\
		&-\tilde{\Delta} f(s)]
		\Big\}ds
		- \int_{t}^{T} \hat{z}^\e_2 (s) dW(s),
	\end{aligned}
\end{equation}
where  $e(s), r(s)$ are defined in \eqref{ee} and
\begin{align}
&\Pi^\e_2(T):=  h\big(X^\e(T)\big) - h(T) - \bar{H}\big(X_1(T)+X_2(T)\big)- \frac{1}{2} X_1(T)^\top H X_1(T),
\label{pi2}\\
&f^\e_{xx}(s):=2\int_{0}^{1}(1-\theta)f_{xx}\Big(s, x^*(s)+\theta\big(x_1(s)+x_2(s)\big),x_\delta^*(s)+\theta\big(x_{\delta, 1}(s)+x_{\delta, 2}(s)\big), \tilde{x}^*(s)+\theta\big(\tilde{x}_1(s)+\tilde{x}_2(s)\big), \nonumber\\
&\q\q\q\q\q
 y^*(s)+\theta\hat{y}^\e(s), z^*(s)+\theta\hat{z}^\e(s), u^*(s), \mu^*(s) \Big)d\theta,\label{fzz}
\end{align}
and $f^\e_{xx_\delta}(s), f^\e_{x\tilde{x}}(s), f^\e_{xy}(s), f^\e_{xz}(s), f^\e_{x_\delta\tilde{x}}(s), f^\e_{x_\delta y}(s), f^\e_{x_\delta z}(s), f^\e_{\tilde{x}y}(s), f^\e_{yz}(s)$ are defined similarly.
Since the BSVIE only admits an $L^2_\dbF$-solution, we need to introduce a BSDE and  handle the terms $\frac{1}{2}f^\e_{yy}e^2$ and $\frac{1}{2}f^\e_{zz}r^2$ in \autoref{important}. Let $(\bar{y}(\cd), \bar{z}(\cd))$ be the solution of the following BSDE.
\begin{align}\label{s2}
\bar{y}^\e(t) =& \int_{t}^{T}\Big\{f_y(s)\bar{y}^\e(s)+f_z(s)\bar{z}^\e(s)+\frac{1}{2}f^\e_{yy}(s)\big(e(s)\big)^2+\frac{1}{2}f^\e_{zz}\big(r(s)\big)^2 + f^\e_{yz}(s)e(s)r(s) \Big\}ds \nonumber\\
&
- \int_{t}^{T}\bar{z}^\e(s)dW(s).
\end{align}
Subtracting \eqref{s2} from \eqref{y2-2}, applying It$\hat{\rm o}$'s formula to $\big(\Gamma^0(s)\big[\hat{y}^\e_2(s)-\bar{y}^\e(s)\big]\big)_{s\in[0, T]}$, one can obtain 
\begin{equation}\label{ie}
\begin{aligned}
&\hat{y}_2^\e(0)-\bar{y}^\e(0) = \dbE\Big[\Gamma^0(T)\Pi^\e_2(T) + \Gamma^0(T)\bar{H}\big(X_1(T)+X_2(T)\big)+ \frac{1}{2} X_1(T)^\top \Gamma^0(T)H X_1(T)
+ \int_{0}^{T}  \Big\{\Gamma^0(s)\bar{F}(s)  \\
& \big(X_1(s)+X_2(s)\big)+ \frac{1}{2}X_1(s)^\top  \big\{\mathcal{A}(s)+\mathcal{A}(s)^\top\big\} X_1(s) 
+\frac{1}{2}\mathcal{X}_1(T, s)^\top \big\{2\lambda^0(s)F_{1y}(s)+2[-f_z(s)\lambda^0(s)+\nu^0(s)] \\
&F_{1z}(s)\big\}X_1(s) + \frac{1}{2}\int_{s}^{T}\mathcal{X}_1(\theta, s)^\top\big\{2\dbE_s[\eta^0(\theta)]F_{1y}(s)+2\big(\zeta^0(\theta, s)-f_z(s)\dbE_s[\eta^0(\theta)]\big)F_{1z}(s) \big\}X_1(s)d\theta 
\\
 &-\Gamma^0(s)
 [I(s) -\tilde{\Delta} f(s)] + \Gamma^0(s)\Big[V_1(s)+\Lambda_1^\e(s)+\Lambda^\e_2(s)\Big]
\Big\}ds\Big],
\end{aligned}\end{equation}
where
\begin{equation}
F_{1y}(s):= 
\begin{bmatrix}
f_{xy}(s) & f_{x_\delta y}(s) &f_{\tilde{x}y}(s)
\end{bmatrix},\q
F_{1z}(s):=
\begin{bmatrix}
f_{xz}(s) & f_{x_\delta z}(s) &f_{\tilde{x}z}(s)
\end{bmatrix}
\end{equation}
\begin{align}
&\mathcal{A}(s):=\frac{1}{2}\Gamma^0(s)F(s)+\int_{s}^{T}C(\theta, s)^\top\dbE_s[\eta^0(\theta)] F_{1z}(s)d\theta+  C(T,s)^\top\lambda_0(s)F_{1z}(s),\\
&\Lambda_1^\e(s):=\frac{1}{2}f^\e_{yy}(s)\Big[|\hat{y}_1^\e(s)|^2 + 2e(s)\hat{y}_1^\e(s)\Big]  +  \frac{1}{2}f^\e_{zz}(s)\Big[|\hat{z}_1^\e(s)|^2 + 2r(s)\hat{z}_1^\e(s)\Big] +f_{yz}^\e(s)\big\{\hat{y}^\e_1(s)\hat{z}_1^\e(s) + e(s)\hat{z}^\e_1(s)\nonumber\\
&\q\q\q\q
+r(s)\hat{y}^\e_1(s)\big\}+f^\e_{xy}(s)\big(x_1(s)+x_2(s)\big)\hat{y}_1^\e(s)+ f^\e_{x_\delta y}(s)\big(x_{\delta, 1}(s)+x_{\delta, 2}(s)\big)\hat{y}_1^\e(s)+f^\e_{\tilde{x}y}(s)\big(\tilde{x}_1(s)\nonumber\\
&\q\q\q\q
+\tilde{x}_2(s)\big)\hat{y}_1^\e(s) +f^\e_{xz}(s)\big(x_1(s)+x_2(s)\big)\hat{z}^\e_1(s)+f^\e_{x_\delta z}(s)\big(x_{\delta,1}(s)+x_{\delta, 2}(s)\big)\hat{z}^\e_1(s)+f^\e_{\tilde{x}z}(s)\big(\tilde{x}_1(s)\nonumber\\
&\q\q\q\q
+\tilde{x}_2(s)\big)\hat{z}^\e_1(s),\label{lam1}
\\
&
\Lambda_2^\e(s):= \frac{1}{2}\Big\{f^\e_{xx}(s)-f_{xx}(s)\Big\}\big(x_1(s)\big)^2 + \frac{1}{2}f^\e_{xx}(s)\Big\{\big(x_2(s)\big)^2+2x_1(s)x_2(s)\Big\} + \frac{1}{2}\Big\{f^\e_{x_\delta x_\delta}(s)-f_{x_\delta x_\delta}(s)\Big\}\nonumber\\
&\q\q\q\q
\big(x_{\delta, 1}(s)\big)^2 + \frac{1}{2}f^\e_{x_\delta x_\delta}(s)\Big\{\big(x_{\delta, 2}(s)\big)^2+2x_{\delta, 1}(s)x_{\delta, 2}(s)\Big\}+ \frac{1}{2}\Big\{f^\e_{\tilde{x}\tilde{x}}(s)-f_{\tilde{x}\tilde{x}}(s)\Big\}\big(\tilde{x}_1(s)\big)^2 + \frac{1}{2}f^\e_{\tilde{x}\tilde{x}}(s)\nonumber\\
&\q\q\q\q\Big\{\big(\tilde{x}_2(s)\big)^2
+2\tilde{x}_1(s)\tilde{x}_2(s)\Big\}+ \Big\{f^\e_{xx_\delta}(s)-f_{xx_\delta}(s)\Big\}x_1(s)x_{\delta, 1}(s)
+ f^\e_{xx_\delta}(s)\Big\{x_1(s)x_{\delta, 2}(s)+x_2(s)x_{\delta, 1}(s)\nonumber\\
&\q\q\q\q
+x_2(s)x_{\delta, 2}(s)\Big\} + \Big\{f^\e_{x\tilde{x}}(s)-f_{x\tilde{x}}(s)\Big\}x_1(s)\tilde{x}_{1}(s)
+ f^\e_{x\tilde{x}}(s)\Big\{x_1(s)\tilde{x}_{2}(s)
+x_2(s)\tilde{x}_{1}(s)+x_{2}(s)\tilde{x}_{ 2}(s)\Big\}\nonumber\\
&\q\q\q\q
+\Big\{f^\e_{x_\delta\tilde{x}}(s)-f_{x_\delta\tilde{x}}(s)\Big\}x_{\delta, 1}(s)\tilde{x}_{1}(s)
+ f^\e_{x_\delta\tilde{x}}(s)\Big\{x_{\delta, 1}(s)\tilde{x}_{2}(s)
+x_{\delta, 2}(s)\tilde{x}_{1}(s)+x_{\delta, 2}(s)\tilde{x}_{ 2}(s)\Big\}+e_1(s)\nonumber\\
&\q\q\q\q
\Big\{[f^\e_{xy}(s) - f_{xy}(s)]x_1(s)+[f^\e_{x_\delta y}(s) - f_{x_\delta y}(s)]x_{\delta, 1}(s)+ [f^\e_{\tilde{x}y}(s) - f_{\tilde{x}y}(s)]\tilde{x}_1(s)\Big\}+e(s)\big[f^\e_{xy}(s)x_2(s)\nonumber\\
&\q\q\q\q
  +f^\e_{x_\delta y}(s)x_{\delta, 2}(s)+f^\e_{\tilde{x}y}(s)\tilde{x}_2(s)\big]+e_2(s)\big[f^\e_{xy}(s)x_1(s)+f^\e_{x_\delta y}(s)x_{\delta, 1}(s)+f^\e_{\tilde{x}y}(s)\tilde{x}_1(s)\big]+ \Gamma^0(s)^{-1}\nonumber\\
&\q\q\q\q r^0_{11}(s)\Big\{[f^\e_{x z}(s) - f_{x z}(s)]x_{1}(s)+[f^\e_{x_\delta z}(s) - f_{x_\delta z}(s)]x_{\delta, 1}(s)+[f^\e_{\tilde{x} z}(s) - f_{\tilde{x} z}(s)]\tilde{x}_{1}(s)\Big\}+ r(s)\Big\{f^\e_{x z}(s)\nonumber \\
&\q\q\q\q x_2(s)+f^\e_{x_\delta z}(s)x_{\delta, 2}(s)+f^\e_{\tilde{x}z}(s)\tilde{x}_2(s)\Big\}
+ \big[f^\e_{xz}(s)
x_1(s)
+f^\e_{x_\delta z}(s)x_{\delta, 1}(s)+ f^\e_{\tilde{x}z}(s)\tilde{x}_1(s)\big]\Gamma^0(s)^{-1}\nonumber\\
&\q\q\q\q\big(r^0_{12}(s)+r^0_2(s)\big).\label{lam2}
\end{align}
Next, we introduce the first-order adjoint equation and the second-order adjoint equation to eliminate the terms $X_1, X_2, X_1^\top X_1$ in \eqref{ie} and explain how we determine  $I(s)$ through two steps. To maintain the continuity of the derivation, the adaptedness and existence and uniqueness of solutions to the adjoint equations will be established in \autoref{Ad-uniquee}, although these properties have already been implicitly used in the present derivation.\\

{\bf Step 1:} in this step, we present the first-order adjoint equation  and use it to eliminate the term $X_1+X_2$. The following is the first-order adjoint equation:
\begin{equation}\label{aad}\left\{\begin{aligned}
		%
		&(a)\q \tilde{\lambda}(t) = \big(\Gamma^0(T)\bar{H}\big)^\top - \int_{t}^{T}\tilde{\nu}(s)dW(s), \q t\in [0, T];\\
		&(b)\q \tilde{\eta}(t) = \big(\Gamma^0(t)\bar{F}(t)\big)^{\top} + A(T, t)^\top\big(\Gamma^0(T)\bar{H}\big)^{\top} + C(T,t)^\top \tilde{\nu}(t) 
		\\
		&\q\q\q\q\q+ \int_{t}^{T}\Big[A(s,t)^\top \tilde{\eta}(s)+ C(s,t)^\top \tilde{\zeta}(s,t)\Big]ds - \int_{t}^{T}\tilde{\zeta}(t,s)dW(s), \q t\in [0, T];\\
		&(c) \q \tilde{\eta}(t) = \dbE[\tilde{\eta}(t)] +  \int_{0}^{t} \tilde{\zeta}(t,s)dW(s) ,\q t\in [0, T].
	\end{aligned}\right.\end{equation}
Let 
\begin{align*}
\psi(t) := \int_{0}^{t}\Big\{B(t, s)+\bar{B}(t, s)\Big\} ds + \int_{0}^{t}\Big\{D(t, s)+\bar{D}(t, s)\Big\}dW(s).
\end{align*}
Then, by following the deduction of Yong \cite[Theorem 5.1]{Yong2008}, we deduce that
\begin{align}
&\dbE\Big[\Gamma^0(T)\bar{H}\big(X_1(T)+X_2(T)\big) + \int_{0}^{T}\Gamma^0(s)\bar{F}(s)\big(X_1(s)+X_2(s)\big)ds\Big] - \dbE\Big[ \Gamma^0(T)\bar{H}\psi(T)\Big]\nonumber\\
&
= \dbE\Big[\int_{0}^{T}\Gamma^0(s)\bar{F}(s)\big(X_1(s)+X_2(s)\big)ds\Big] + \dbE\Big[\Gamma^0(T)\bar{H}  \Big(X_1(T)+X_2(T)- \psi(T) \Big)\Big]\nonumber \\
&
=\dbE\Big[\int_{0}^{T}\Gamma^0(s)\bar{F}(s)\big(X_1(s)+X_2(s)\big)ds + \Gamma^0(T)\bar{H} \Big( \int_{0}^{T}A(T, s) \big(X_1(s)+X_2(s)\big)ds \\
&\q\q
+ \int_{0}^{T}C(T, s)\big(X_1(s)+X_2(s)\big)dW(s)\Big)\Big] 
\nonumber\\
&
= \dbE\Big[\int_{0}^{T}\big(X_1(s)+X_2(s)\big)^\top\big(\Gamma^0(s)\bar{F}(s)\big)^\top ds + \int_{0}^{T}\big(X_1(s)+X_2(s)\big)^\top A(T, s)^\top\big(\Gamma^0(T)\bar{H}\big)^{\top}ds\nonumber\\ &
\q\q\q+ \Big(\tilde{\lambda}(0)^\top + \int_{0}^{T}\tilde{\nu}(s)dW(s)\Big)\int_{0}^{T}C(T, s)\big(X_1(s)+X_2(s)\big)dW(s)\Big] \nonumber\\
&=\dbE\Big[\int_{0}^{T}\big(X_1(s)+X_2(s)\big)^\top\big(\Gamma^0(s)\bar{F}(s)\big)^\top ds + \int_{0}^{T}\big(X_1(s)+X_2(s)\big)^\top A(T, s)^\top\big(\Gamma^0(T)\bar{H}\big)^{\top}ds\nonumber\\ &
\q\q\q+  \int_{0}^{T}\big(X_1(s)+X_2(s)\big)^\top C(T, s)^\top\tilde{\nu}(s)ds\Big] \nonumber\\
&
= \dbE\Big[\int_{0}^{T}\big(X_1(s)+X_2(s)\big)^\top\Big( \tilde{\eta}(s) -  \int_{s}^{T}\Big[A(t, s)^\top\tilde{ \eta}(t)+C(t, s)^\top \tilde{\zeta}(t, s)\Big]dt \Big)ds\Big] \nonumber\\
&
=\dbE\Big[\int_{0}^{T} \tilde{\eta}(s)^\top \Big(X_1(s)+X_2(s) - \int_{0}^{s}A(s, t)\big(X_1(t)+X_2(t)\big)dt\Big)ds\Big] \nonumber\\
&\q\q
 - \dbE\Big[\int_{0}^{T}\big(X_1(s)+X_2(s)\big)^\top \int_{s}^{T}C(t, s)^\top\tilde{\zeta}(t, s)dtds\Big] \nonumber\\
&
=\dbE\Big[\int_{0}^{T} \tilde{\eta}(s)^\top \Big(X_1(s)+X_2(s) - \int_{0}^{s}A(s, t)\big(X_1(t)+X_2(t)\big)dt\Big)ds\Big] \nonumber\\
&\q\q- \dbE\Big[\int_{0}^{T} \int_{0}^{s} C(s, t)\big(X_1(t)+X_2(t)\big)
\tilde{\zeta}(s, t)^\top dtds\Big]  \nonumber\\
&
=\dbE\Big[\int_{0}^{T} \tilde{\eta}(s)^\top \Big(X_1(s)+X_2(s) - \int_{0}^{s}A(s, t)\big(X_1(t)+X_2(t)\big)dt\Big)ds\Big] \nonumber\\
&\q\q- \dbE\Big[\int_{0}^{T} \int_{0}^{s} C(s, t)\big(X_1(t)+X_2(t)\big)
dW(t) \int_{0}^{s}\tilde{\zeta}(s, t)^\top dW(t)ds\Big]\nonumber\\
&
=\dbE\Big[\int_{0}^{T} \tilde{\eta}(s)^\top \Big(X_1(s)+X_2(s) - \int_{0}^{s}A(s, t)\big(X_1(t)+X_2(t)\big)dt\Big)ds\Big] \nonumber\\
&\q\q- \dbE\Big[\int_{0}^{T} \int_{0}^{s} C(s, t)\big(X_1(t)+X_2(t)\big)
dW(t) \Big(\tilde{\eta}(s)^\top - \dbE[\tilde{\eta}(s)]^\top\Big)\Big]\nonumber\\
&
=\dbE\Big[\int_{0}^{T} \tilde{\eta}(s)^\top \Big(X_1(s)+X_2(s) - \int_{0}^{s}A(s, t)\big(X_1(t)+X_2(t)\big)dt - \int_{0}^{s} C(s, t)\big(X_1(t)+X_2(t)\big)dW(t)\Big)ds\Big]\nonumber\\
&
= \dbE\Big[\int_{0}^{T}\tilde{\eta}(s)^\top \psi(s)ds\Big].
\end{align}
Thus, we obtain that
\begin{align}
	&\dbE\Big[\Gamma^0(T)\bar{H}\big(X_1(T)+X_2(T)\big) + \int_{0}^{T}\Gamma^0(s)\bar{F}(s)\big(X_1(s)+X_2(s)\big)ds\Big] \nonumber\\
	&
	= \dbE\Big[ \Gamma^0(T)\bar{H}\psi(T)\Big] + \dbE\Big[\int_{0}^{T}\tilde{\eta}(s)^\top \psi(s)ds\Big].
\end{align}
Let
\begin{equation} \label{qq}
	\tilde{\lambda}(s) \triangleq 
	\begin{bmatrix}
		&\tilde{\lambda}^{0}(s) \\
		& \tilde{\lambda}^{1}(s)\\
		& \tilde{\lambda}^{2}(s)
	\end{bmatrix},\q
	\tilde{\nu}(s) \triangleq
	\begin{bmatrix}
		&\tilde{\nu}^{0}(s)\\
		&\tilde{\nu}^{1}(s)\\
		&\tilde{\nu}^{2}(s)
	\end{bmatrix},\q
	\tilde{\eta}(s) \triangleq 
	\begin{bmatrix}
		&\tilde{\eta}^{0}(s) \\
		& \tilde{\eta}^{1}(s)\\
		& \tilde{\eta}^{2}(s)
	\end{bmatrix},\q
	\tilde{\zeta}(t, s) \triangleq
	\begin{bmatrix}
		&\tilde{\zeta}^{0}(t, s)\\
		&\tilde{\zeta}^{1}(t, s)\\
		&\tilde{\zeta}^{2}(t, s)
	\end{bmatrix}.
\end{equation}
Then, by the definition of $B(t, s), D(t, s), \bar{B}(t, s), \bar{D}(t, s)$ and \eqref{adjointe}(a)(c) and note the fact that
$\int_{0}^{T}\int_{0}^{s}a(s, t)dtds = \int_{0}^{T}\int_{t}^{T}a(s, t)dsdt = \int_{0}^{T}\int_{s}^{T}a(t, s)dtds$ and $\tilde{\lambda}_0(s) = \dbE_s[\Gamma^0(T)h_x\big(x^*(T)\big)]$,
 we deduce that
\begin{align}
&\dbE\Big[ \Gamma^0(T)\bar{H}\psi(T)\Big] + \dbE\Big[\int_{0}^{T}\tilde{\eta}(s)^\top \psi(s)ds\Big] \nonumber\\
&= \dbE\Big[ \Gamma^0(T)\bar{H}\Big(\int_{0}^{T}\Big\{B(T, s)+\bar{B}(T, s)\Big\}ds + \int_{0}^{T}\Big\{D(T, s)+\bar{D}(T, s)\Big\}dW(s)\Big)\Big] \nonumber\\
&\q\q
+ \dbE\Big[\int_{0}^{T}\tilde{\eta}(s)^\top \Big(\int_{0}^{s}\Big\{B(s, t)+\bar{B}(s, t)\Big\}dt + \int_{0}^{s}\Big\{D(s, t)+\bar{D}(s, t)\Big\}dW(t)\Big)ds\Big]\nonumber\\
&
=\dbE\Big[ \int_{0}^{T}\Gamma^0(T)\bar{H}\Big\{B(T, s)+\bar{B}(T, s)\Big\}ds + \int_{0}^{T}\tilde{\nu}(s)\Big\{D(T, s)+\bar{D}(T, s)\Big\}ds\Big)\Big] \nonumber\\
&\q\q+ \dbE\Big[\int_{0}^{T} \int_{0}^{s}\tilde{\eta}(s)^\top\Big\{B(s, t)+\bar{B}(s, t)\Big\}dtds\Big]
 +\dbE\Big[\int_{0}^{T} \int_{0}^{s}\tilde{\zeta}(s, t)^\top\Big\{D(s,t)+\bar{D}(s, t)\Big\}dtds\Big]\nonumber\\
 & = \dbE\Big[\int_{0}^{T}\Big(\frac{1}{2}X_1(s)^\top \partial^2b(s)X_1(s)+\Delta b(s)\Big)\Gamma^0(T)\{h_x(T)+h_{x_\delta}(T)1_{(0, T-\delta)}(s)\}ds  \Big] + \dbE\Big[\int_{0}^{T}\Big\{\Big(\frac{1}{2}X_1(s)^\top  \nonumber\\
 &\q
 \partial^2\sigma(s)(s)X_1(s)+ \Delta \Xi(s)X_1(s)+\Delta\sigma(s) \Big)\tilde{\nu}^0(s) + \Big(\frac{1}{2}X_1(s)^\top \partial^2\sigma(s)(s)X_1(s)+ \Delta \Xi(s)X_1(s) +\Delta \sigma(s) \Big)\nonumber \\
 &\q\tilde{\nu}^1(s)1_{(0, T-\delta)}(s)\Big\}ds\Big]+
 \dbE\Big[\int_{0}^{T}\Big(\frac{1}{2}X_1(s)^\top \partial^2b(s)X_1(s) + \Delta b(s)\Big)\int_{s}^{T}\Big\{ \tilde{\eta}^0(t)+ \tilde{\eta}^1(t)1_{(0, t-\delta)}(s)\Big\}dtds\Big]\nonumber\\
 &
 \q +
 \dbE\Big[\int_{0}^{T}\Big(\frac{1}{2}X_1(s)^\top \partial^2\sigma(s)X_1(s) + \Delta\Xi(s)X_1(s)+\Delta\sigma(s)\Big)\int_{s}^{T}\Big\{ \tilde{\zeta}^0(t, s) + 
 \tilde{\zeta}^1(t, s)1_{(0, t-\delta)}(s)\Big\}dtds\Big]\nonumber\\
 &= 
 \dbE\Big[\int_{0}^{T}\frac{1}{2}X_1(s)^\top \partial^2b(s)X_1(s)p(s)ds \Big]+ \dbE\Big[\int_{0}^{T}\frac{1}{2}X_1(s)^\top \partial^2\sigma(s)X_1(s)q(s) ds\Big] + \dbE\Big[\int_{0}^{T}\Delta b(s)p(s) ds\Big] \nonumber\\
 &\q
 + \dbE\Big[\int_{0}^{T}\Delta \sigma(s)q(s) ds\Big]+ \dbE\Big[\int_{0}^{T}\Delta \Xi(s)X_1(s)q(s) ds\Big], \label{x2}
\end{align}
where
\begin{equation}\label{pq}
\begin{aligned}
&p(s):= \dbE_s\Big[\int_{s}^{T}\Big\{\tilde{ \eta}^0(t)+ \tilde{\eta}^1(t)1_{(0, t-\delta)}(s)\Big\}dt\Big] + \tilde{\lambda}^0(s) + \tilde{\lambda}^1(s)1_{(0, T-\delta)}(s),\\
&q(s):= \dbE_s\Big[\int_{s}^{T}\Big\{ \tilde{\zeta}^0(t, s) + 
\tilde{\zeta}^1(t, s)1_{(0, t-\delta)}(s)\Big\}dt\Big] + \nu^0(s)+\tilde{\nu}^1(s)1_{(0, T-\delta)}(s).
\end{aligned}\end{equation}
Substituting \eqref{x2} into \eqref{ie},  we obtain that
\begin{equation}\label{y4}
	\begin{aligned}
		&\hat{y}_2^\e(0)-\bar{y}^\e(0) = \dbE\Big[\Gamma^0(T)\Pi^\e_2(T) + \frac{1}{2} X_1(T)^\top \Gamma^0(T)H X_1(T)
		+ \int_{0}^{T}  \Big\{  \frac{1}{2}\mathcal{X}_1(T, s)^\top M_1(s)X_1(s)+ \frac{1}{2} X_1(s)^\top \\
		&
		M_2(s)X_1(s)+ \frac{1}{2}\int_{s}^{T}\mathcal{X}_1(\theta, s)^\top M_3(\theta, s)X_1(s)d\theta-\Gamma^0(s) \Big(I(s)- \Gamma^0(s)^{-1}\Delta b(s)p(s) - \Gamma^0(s)^{-1}\Delta\sigma(s)q(s)\\
		&-\tilde{\Delta} f(s)\Big)+ \Delta\Xi(s)X_1(s)q(s)+ \Gamma^0(s)\Big(V_1(s)+\Lambda_1^\e(s)
		+\Lambda^\e_2(s)\Big)
		\Big\}ds\Big],
	\end{aligned}
\end{equation}
where
\begin{equation}\label{m}
\begin{aligned}
&M_1(s):=2\lambda^0(s)F_{1y}(s)+2\big\{-f_z(s)\lambda^0(s)+\nu^0(s)\big\}F_{1z}(s),\\
& M_2(s):=\mathcal{A}(s)+\mathcal{A}(s)^\top + p(s)\partial^2 b(s) 
+ q(s)\partial^2 \sigma(s),
\\
&M_3(\theta, s):= 2\dbE_s[\eta^0(\theta)]F_{1y}(s)
+2\big\{\zeta^0(\theta, s)-f_z(s)\dbE_s[\eta^0(\theta)]\big\}F_{1z}(s).
\end{aligned}\end{equation}
{\bf Step 2:} in this step, by following the approach of Meng--Shi--Wang--Zhang  \cite[Lemma 4.2]{meng2025general}, we present the second-order adjoint equation and use it to eliminate the term $X_1^\top X_1$.
The following is the second-order adjoint equation:
\begin{equation}\label{Adjointe}\left\{\begin{aligned}
		&(a)\q P_1(t) = \Gamma^0(T)H - \int_{t}^{T}Q_1(\theta)dW(\theta), \q 0\leq t \leq T;\\
		&(b)\q P_2(t) =  A(T, t)^\top P_1(t)  + C(T, t)^\top Q_1(t)  +\frac{1}{2}M_1(t)^\top \\
		&\q\q\q\q\q+ \int_{t}^{T}\Big[A(\theta,t)^\top P_2(\theta)+ C(\theta,t)^\top Q_2(\theta,t)\Big]d\theta - \int_{t}^{T}Q_2(t,\theta)dW(\theta), \q 0\leq t \leq T;\\
		&(c) \q P_3(t) = M_2(t) +  C(T,t)^\top P_1(t) C(T,t) \\
		&\q\q\q +\int_{t}^{T}\Big[C(T,t)^\top P_2(\theta)^{\top} C(\theta,t)+ C(\theta,t)^\top P_2(\theta) C(T,t) + C(\theta,t)^\top P_3(\theta) C(\theta,t)\Big]d\theta \\
		&\q\q\q
		+\int_{t}^{T}\int_{t}^{T}C(\theta,t)^\top P_4(\theta', \theta)C(\theta',t)d\theta d\theta'-\int_{t}^{T}Q_3(t,\theta)dW(\theta), \q 0\leq t \leq T;\\
		&(d) \q P_4(\theta, t) = \frac{1}{2}M_3(\theta, t)^\top+A(T,t)^\top P_2(\theta)^\top + C(T,t)^\top Q_2(\theta,t)^\top+ A(\theta,t)^\top P_3(\theta) + C(\theta, t)^\top Q_3(\theta,t)\\
		&\q\q\q\q + \int_{t}^{T}\Big[C(\theta',t)^\top Q_4(\theta, \theta', t) + A(\theta', t)^\top P_4(\theta,\theta')\Big]d\theta'  - \int_{t}^{T}Q_4(\theta,t,\theta')dW(\theta'), \q0\leq t \leq \theta \leq T;\\
		&(e)\q P_4(\theta,t) = P_4(t,\theta)^\top, \q Q_4(\theta, t, \theta') = Q_4(t,\theta, \theta')^\top, \q 0\leq \theta \leq t\leq T
	\end{aligned}\right.\end{equation}
subject to the following constraints:
\begin{equation}\label{subjecte}\left\{\begin{aligned}
		&(a)\q P_2(t) = \dbE_{\theta}[P_2(t)]+ \int_{\theta}^{t}Q_2(t,\theta')dW(\theta'), \q 0\leq \theta \leq t \leq T;\\
		&(b)\q P_3(t) = \dbE_{\theta}[P_3(t)]+ \int_{\theta}^{t}Q_3(t,\theta')dW(\theta'), \q 0\leq \theta \leq t \leq T; \\
		&(c) \q P_4(\theta, t) = \dbE_{\theta'}[P_4(\theta, t)]+ \int_{\theta'}^{t\wedge \theta}Q_4(\theta, t,s)dW(s), \q 0\leq \theta' \leq (\theta \wedge t) \leq T. \\
	\end{aligned}\right.\end{equation}
Recall that $	\mathcal{X}_1(s, \theta)$ and its property is defined in \eqref{xx} and presented in \eqref{pro}.

First, we handle the term $\dbE[X_1(T)^\top \Gamma(T)HX_1(T)]$. Applying It$\hat{\rm o}$'s formula to $$\big(\mathcal{X}_1^\top (T, t) P_1(t)\mathcal{X}_1 (T, t)\big)_{t \in [0, T]}$$ and by the fact $H^\top = H$ and \autoref{state}, one can deduce that
\begin{align}
&\dbE[X_1(T)^\top \Gamma(T)HX_1(T)] = \dbE[\mathcal{X}_1(T, T)^\top P_1(T)\mathcal{X}_1(T, T)] \nonumber\\
&
= \dbE\Big[\int_{0}^{T}\Big\{\mathcal{X}_1(T, t)^\top [P_1(t)A(T, t) + Q_1(t) C(T, t)]X_1(t) + X_1(t)^\top [A(T, t) ^\top P_1(t)+ C(T, t)^\top Q_1(t) ]\mathcal{X}_1(T, t) \nonumber\\
&\q\q
+ X_1(t)^\top C(T, t)^\top P_1(t)C(T, t)X_1(t)+ D(T, t)^\top P_1(t)D(T, t) \Big\}dt\Big] + o(\e).\label{k1}
\end{align}
Second, we handle the term $\mathcal{X}_1(T, t)^\top[...]X_1(t)$. In view of \eqref{subjecte}(a) and applying It$\hat{\rm o}$'s formula to $$\big(\mathcal{X}_1(T,\theta)^\top\dbE_\theta[P_2(t)]\mathcal{X}_1(t, \theta)\big)_{\theta \in [0, t]} = \Big(\mathcal{X}_1(T,\theta)^\top\Big\{P_2(t) - \int_{\theta}^{t}Q_2(t, k)dW(k)\Big\}^\top \mathcal{X}_1(t, \theta)\Big)_{\theta\in[0, t]},$$
and by the fact that $\int_{0}^{T}\int_{0}^{s}a(s, t)dtds  =  \int_{0}^{T}\int_{s}^{T}a(t, s)dtds$ and  \eqref{Adjointe}(b), we deduce that
\begin{align}
&\dbE\Big[\int_{0}^{T}\Big\{\mathcal{X}_1(T, t)^\top \big\{ [P_1(t)A(T, t) + Q_1(t) C(T, t)]
+\frac{1}{2}M_1(t)\big\}X_1(t)\Big\}dt\Big]\nonumber\\
& + \dbE\Big[\int_{0}^{T}\Big\{X_1(t)^\top \big\{ [A(T, t)^\top P_1(t) + C(T, t)^\top Q_1(t) ]
+\frac{1}{2}M_1(t)^\top\big\}\mathcal{X}_1(T, t)\Big\}dt\Big]\nonumber\\
&
= \dbE\bigg[\int_{0}^{T}\bigg\{X_1(t)^\top \int_{t}^{T}\Big(C(T, t)^\top P_2(\theta)^\top C(\theta, t) + C(\theta, t)^\top P_2(\theta)C(T, t)\Big)d\theta X_1(t) + \int_{t}^{T}\Big[X_1(t)^\top\nonumber\\
&\q\q\q
 \Big(P_2(\theta)A(T, t) + Q_2(\theta, t)C(T, t)\Big)^\top \mathcal{X}_1(\theta, t) + \mathcal{X}_1(\theta, t)^\top  \Big(P_2(\theta)A(T, t) + Q_2(\theta, t)C(T, t)\Big)X_1(t)  \Big]d\theta \nonumber\\
 &\q\q\q+ \int_{t}^{T} \Big[D(T, t)^\top P_2(\theta)^\top D(\theta, t) + D(\theta, t)^\top P_2(\theta)D(T, t)\Big]d\theta
\bigg\}dt\bigg] + o(\e).\label{k2}
\end{align}

Third, we handle the term of $X_1(t)^\top[...]X_1(t)$, $\mathcal{X}_1(\theta, t)^\top[...]X_1(t)$, and their transpose. 
In view of \eqref{subjecte}(b) and applying It$\hat{\rm o}$'s formula to $$\big(\mathcal{X}_1(t,\theta)^\top\dbE_\theta[P_3(t)]\mathcal{X}_1(t, \theta)\big)_{\theta \in [0, t]} = \Big(\mathcal{X}_1(t,\theta)\Big\{P_3(t) - \int_{\theta}^{t}Q_3(t, k)dW(k)\Big\}\mathcal{X}_1(t, \theta)\Big)_{\theta\in[0, t]},$$
one can deduce that
\begin{align}\label{k6}
&\dbE\bigg[X_1(t)^\top \Big\{M_2(t)
 + C(T, t)^\top P_1(t)C(T, t)+ \int_{t}^{T}\Big(C(T, t)^\top P_2(\theta)^\top C(\theta, t) + C(\theta, t)^\top P_2(\theta)C(T, t)\Big)d\theta \Big\}X_1(t)\bigg]\nonumber\\
 &= \dbE\Big[X_1(t)^\top P_3(t) X_1(t)\Big]  - \dbE\Big[X_1(t)^\top \int_{t}^{T}C(\theta, t)^\top P_3(\theta)C(\theta, t)d\theta X_1(t)\Big] \nonumber\\
 &\q\q
 - \dbE\Big[X_1(t)^\top\int_{t}^{T}\int_{t}^{T}C(\theta, t)^\top P_4(\theta', \theta)C(\theta', t)d\theta d\theta'X_1(t)\Big]\nonumber\\
 &
 = \dbE\bigg[\int_{t}^{T}\bigg\{X_1(t)^\top\Big[A(\theta, t)^\top P_3(\theta) + C(\theta, t)^\top Q_3(\theta, t)\Big]\mathcal{X}_1(\theta, t)  \nonumber\\
 &\q\q\q\q
 + \mathcal{X}_1(\theta, t)^\top\Big[A(\theta, t)^\top P_3(\theta)^\top+ C(\theta, t)^\top Q_3(\theta, t)^\top\Big]^\top X_1(t) + D(\theta, t)^\top P_3(\theta)D(\theta, t)\bigg\}d\theta\bigg]\nonumber\\
 &\q\q\q\q
 - \dbE\Big[X_1(t)^\top\int_{t}^{T}\int_{t}^{T}C(\theta, t)^\top P_4(\theta', \theta)C(\theta', t)d\theta d\theta'X_1(t)\Big].
\end{align}

Finally, applying It$\hat{\rm o}$'s formula to $$\big(\mathcal{X}_1(\theta, \theta')^\top\dbE_{\theta'}[P_4(\theta, t)]\mathcal{X}_1(t, \theta' )\big)_{\theta' \in [0, t]} = \Big(\mathcal{X}_1(\theta, \theta')^\top\Big\{P_4(\theta, t) - \int_{\theta'}^{t}Q_4(\theta, t,  k)dW(k)\Big\}\mathcal{X}_1(t, \theta')\Big)_{\theta'\in[0, t]},$$
and by the fact that
\begin{align*}
&\int_{t}^{T}\int_{t}^{T}C(\theta, t)^\top P_4(\theta', \theta) C(\theta', t)d\theta d\theta' \\
&= \int_{t}^{T}\int_{\theta'}^{T}C(\theta, t)^\top P_4(\theta', \theta) C(\theta', t)d\theta d\theta' + \int_{t}^{T}\int_{\theta'}^{T}C(\theta', t)^\top P_4(\theta', \theta) C(\theta, t)d\theta d\theta',
\end{align*}
one can deduce that
	\begin{align}
		\hat{y}^\e_2(0)- &\bar{y}(0) = \dbE\bigg[ \Gamma^0(T)\Pi^\e_2(T)  
		+ \int_{0}^{T}\bigg\{ 
		\Delta \Xi(t)X_1(t)q(t)
		+\Gamma^0(t)\big[V_1(t)+\Lambda^\e_1(t)
		+ \Lambda_2^\e(t)\big]
		- \Gamma^0(t)\Big\{I(t) -\tilde{\Delta} f(t) 
		 \nonumber \\
		&- \Gamma^0(t)^{-1}\Delta b(t)p(t)- \Gamma^0(t)^{-1}\Delta\sigma(t)q(t)- \frac{1}{2}\Gamma^0(t)^{-1} D(T, t)^\top P_1(t)D(T, t)
		- \frac{1}{2}\int_{t}^{T} \Gamma^0(t)^{-1}  \nonumber\\
		&\big[D(T, t)^\top P_2(\theta)^\top D(\theta, t)+ D(\theta, t)^\top P_2(\theta)D(T, t)
		+ D(\theta, t)^\top P_3(\theta)D(\theta, t)\big]d\theta -\frac{1}{2}\int_{t}^{T}\int_{t}^{T}\Gamma^0(t)^{-1}\nonumber\\
		&D(\theta,t)^\top P_4(\theta', \theta)
		D(\theta', t)d\theta d\theta'\Big\} \bigg\}dt\bigg]. \label{yy}
	\end{align}
Consequently, we can define the variational equation by letting
\begin{align}\label{ii}
&I(t):= \tilde{\Delta} f(t) +
\Gamma^0(t)^{-1}\Delta b(t)p(t)+ \Gamma^0(t)^{-1}\Delta\sigma(t)q(t)+ \dbE_t\Big[\frac{1}{2}\Gamma^0(t)^{-1} D(T, t)^\top P_1(t)D(T, t)
+\frac{1}{2}\int_{t}^{T} \Gamma^0(t)^{-1}  \nonumber\\
&\big[D(T, t)^\top P_2(\theta)^\top D(\theta, t)+ D(\theta, t)^\top P_2(\theta)D(T, t)
+ D(\theta, t)^\top P_3(\theta)D(\theta, t)\big]d\theta +\frac{1}{2}\int_{t}^{T}\int_{t}^{T}\Gamma^0(t)^{-1}D(\theta,t)^\top\nonumber\\
& P_4(\theta', \theta)
D(\theta', t)d\theta d\theta'\Big].
\end{align}
By a similar deduction of \autoref{aua} and a similar proof of Wang--Yong \cite[Theorem 5.1]{wang2023spike}, we can get the following result.
\begin{proposition}\label{Ad-uniquee}
	Let \autoref{A1}  hold. 
	Then, for $\beta>1$, \eqref{aad} admits a unique solution  satisfying $(\tilde{\lambda}(\cd), \tilde{\nu}(\cd)) \in S^{\beta}_{\dbF}(0, T;\dbR^{3}) \times \mathcal{M}^{2, \beta}_{\dbF}(0, T;\dbR^{3})$,
	$(\tilde{\eta}(\cd), \tilde{\zeta}(\cd,\cd)) \in L^{2}_{\dbF}(0, T;\dbR^{3}) \times L^2\big(0, T ; L^2_{\dbF}(0, T; \dbR^{3})\big)$
	and \eqref{Adjointe} admits a unique solution $(P_1(\cd), Q_1(\cd)) \in S^{\beta}_{\dbF}(0, T; \dbR^{3\times 3}) \times \big(\mathcal{M}^{2, \beta}_{\dbF}(0, T; \dbR^{3\times 3})$,
	$\big(P_2(\cd), Q_2(\cd, \cd)\big),  \in L^{2}_{\dbF}(0, T; \dbR^{3\times 3})\times L^2\big(0, T ; L^2_{\dbF}(0, T; \dbR^{3\times 3})\big) $,
	$\big(P_3(\cd), Q_3(\cd, \cd)\big) \in L^{2}_{\dbF}(0, T; \dbR^{3\times 3})\times L^2\big(0, T ; L^2_{\dbF}(0, T; \dbR^{3\times 3})\big),$
	$\big(P_4(\cd,\cd), Q_4(\cd,\cd,\cd)\big) \in L^2\big(0, T ; L^2_{\dbF}(0, T; \dbR^{3\times 3})\big) \times L^2\big(0, T; L^2\big(0, T ; L^2_{\dbF}(0, T; \dbR^{3\times 3})\big) \big)$.
\end{proposition}

Set $\hat{x}^\e_{\delta, 2}(s):= \hat{x}^\e_{\delta, 1} (s)- x_{\delta, 2}(s), \q \hat{\tilde{x}}_2^\e(s):=  \hat{\tilde{x}}_1^\e(s)-\tilde{x}_2(s).$
The following is the estimate for $y^\e(0)-y^*(0) - \hat{y}(0) - \bar{y}(0)$.
\begin{lemma}\label{last1}
	Let \autoref{A1} hold. Then, we have that
	\begin{align*}
		y^\e(0) - y^*(0) -\hat{y}(0) - \bar{y}(0)= o(\e).
	\end{align*}
\end{lemma}
\begin{proof}
	In view of \eqref{yy} and \eqref{ii}, one can obtain that
	\begin{align}
		|\hat{y}^\e_2(0)- \bar{y}(0)|\leq &\dbE\Big[ \Gamma^0(T)|\Pi^\e_2(T) |
		+ \int_{0}^{T}\Big\{ 
		|\Delta \Xi(s)X_1(s)q(s)|
		+\Gamma^0(s)|V_1(s)| +\Gamma^0(s)|\Lambda_1^\e(s)|
		+  \Gamma^0(s)|\Lambda_2^\e(s)|
		\Big\}ds\Big]. 
	\end{align}
	We first estimate the term $\dbE[\int_{0}^{T}\Gamma^0(s)|V_1(s)|ds]$. 
	In view of the definition of $V_1$ in \eqref{v1}, let $E_\e := [t_0, t_0+\e]\cup [t_0+\delta, t_0+\delta + \e]$, then, by H$\ddot{\rm o}$lder's inequality, \autoref{state}, \autoref{e2p}, and \eqref{aa},  one can deduce that
	\begin{align*}
	\dbE\Big[\int_{0}^{T}\Gamma^0(s)|V_1(s)|ds] &\leq K \dbE\Big[\int_{0}^{T}\Gamma^0(s)\big\{|\hat{x}^\e_2(s)| + |\hat{x}_{\delta, 2}^\e(s)| +|\hat{\tilde{x}}^\e_2(s)| \big\}1_{[0, T]\backslash E_\e}(s)ds \Big]+ \dbE\Big[ \int_{t_0}^{t_0+\e}a^\e(s)ds\Big]\\
	&\q+ \dbE\Big[ \int_{t_0+\delta}^{t_0+\delta+\e}|a^\e(s)|ds\Big] = o(\e).
	\end{align*}
where
\begin{align*}
a^\e(s) =& \Gamma^0(s)\big\{|\hat{x}^\e(s)|+ |\hat{x}_\delta^\e(s)|+|\hat{\tilde{x}}^\e(s)|+|x_1(s)|+|x_2(s)|+|x_{\delta, 1}(s)|+|x_{\delta, 2}(s)|+|\tilde{x}_1(s)|+|\tilde{x}_2(s)| \\
&+ |\hat{y}^\e(s)|+|\hat{z}^\e_1(s)|\big\} + |r^0_{11}(s)| + |r^0_2(s)|.
\end{align*}
	
Second, we estimate the term $\dbE[\Gamma^0(T)|\Pi^\e_2(T)|]$. In view of the definition of $\Pi^\e_2$ in \eqref{pi2}, one can obtain that
\begin{align}
\dbE\Big[\Gamma^0(T)|\Pi^\e_2(T)|\Big]\leq K \dbE\Big[\Gamma^0(T)\{|\hat{x}_2^\e(T)|+ |\hat{x}_{\delta, 2}^\e(T)|+|\hat{\tilde{x}}_2^\e(T)|+ b_1+b_2\}\Big],
\end{align}
where 
\begin{align}
&b_1 := |h_{xx}^\e(T) - h_{xx}(T)| \big(\hat{x}^\e(T)\big)^2 + |h_{x_\delta x_\delta}^\e(T) - h_{x_\delta x_\delta}(T)| \big(\hat{x}_\delta^\e(T)\big)^2+|h_{\tilde{x}\tilde{x}}^\e(T) - h_{\tilde{x}\tilde{x}}(T)| \big(\hat{\tilde{x}}^\e(T)\big)^2\nonumber\\
&
+|h_{xx_\delta}^\e(T) - h_{xx_\delta}(T)|\hat{x}^\e(T)\hat{x}^\e_\delta(T)| + |h_{x\tilde{x}}^\e(T) - h_{x\tilde{x}}(T)|\hat{x}^\e(T)\hat{\tilde{x}}^\e(T)| 
+|h_{x_\delta\tilde{x}}^\e(T) - h_{x_\delta\tilde{x}}(T)|\hat{x}_\delta^\e(T)\hat{\tilde{x}}^\e(T)|,\nonumber\\
&
b_2:= |\hat{x}^\e_1(T)\big(\hat{x}^\e(T)+x_1(T)\big)| +
|\hat{x}^\e_{\delta, 1}(T)\big(\hat{x}^\e_\delta(T)+x_{\delta, 1}(T)\big)| + 
|\hat{\tilde{x}}^\e_1(T)\big(\hat{\tilde{x}}^\e(T)+\tilde{x}_1(T)\big)|  + |\hat{x}^\e_1(T)\hat{x}^\e_\delta(T) \nonumber\\
&
+x_1(T)\hat{x}^\e_{\delta, 1}(T)|+ |\hat{x}^\e_1(T)\hat{\tilde{x}}^\e(T)+x_1(T)\hat{\tilde{x}}^\e_{1}(T)|+ |\hat{x}^\e_{\delta, 1}(T)\hat{\tilde{x}}^\e(T)+x_{\delta, 1}(T)\hat{\tilde{x}}_1^\e(T)|,\nonumber\\
&
h_{xx}^\e(T):= 2\int_{0}^{1}(1-\theta)h_{xx}\big(x^*(T)+\theta \hat{x}^\e(T), x_\delta^*(T)+\theta \hat{x}_\delta^\e(T), \tilde{x}^*(T)+\theta \hat{\tilde{x}}^\e(T)\big)d\theta,\nonumber
\end{align}
and $h^\e_{x_\delta x_\delta}(T), h^\e_{\tilde{x}\tilde{x}}(T), h^\e_{x x_\delta}(T), h^\e_{x\tilde{x}}(T), h^\e_{x_\delta \tilde{x}}(T), $ are defined similarly.
It then follows from H$\ddot{\rm o}$lder's inequality, \eqref{state}, and Dominated convergence theorem that
\begin{align}
\dbE\Big[\Gamma^0(T)|\Pi^\e_2(T)|\Big]=o(\e).
\end{align} 

Third, we estimate the terms $\dbE[\int_{0}^{T}\{\Gamma^0(s)|\Lambda_1(s)|+\Gamma^0(s)|\Lambda_2(s)|\}ds]$. Recall that the definition of $\Lambda_1(s), \Lambda_2(s)$ are in \eqref{lam1}, \eqref{lam2}. By H$\ddot{\rm o}$lder's inequality, \autoref{A1}, \autoref{state}, \autoref{ro est}, and \autoref{z1}, one can deduce that 
$\dbE[\int_{0}^{T}\Gamma^0(s)|\Lambda^\e_1(s)|ds] = o(\e)$. In addition, by H$\ddot{\rm o}$lder's inequality, \autoref{state}, \autoref{e2p}, and Dominated convergence theorem, we can deduce that
\begin{align*}
\dbE\Big[\int_{0}^{T}r^0_{11}(s)[f^\e_{x z}(s) - f_{x z}(s)]x_1(s)ds\Big] = o(\e).
\end{align*}
The other terms of $\dbE[\int_{0}^{T}\Gamma^0(s)|\Lambda^\e_2(s)|ds]$ can be estimated similarly.

Finally, we estimate the term $\dbE[\int_{0}^{T}\Delta \Xi(s)X_1(s)q(s)] = o(\e)$.
Recall that 
\begin{equation}
\Delta\Xi(s):=
\begin{bmatrix}
\Delta\sigma_x(s)& \Delta\sigma_{x_\delta}(s)& \Delta\sigma_{\tilde{x}}(s)
\end{bmatrix}.
\end{equation}
Then, by \autoref{A1}, H$\ddot{\rm o}$lder's inequality, \autoref{Ad-uniquee}, \autoref{state}, and \autoref{Lebes}, one can deduce that
$\dbE[\int_{0}^{T}\Delta \Xi(s)X_1(s)q(s)ds] = o(\e).$

This completes the proof.
\end{proof}

\section{Analysis of duality relations for the terms $\frac{1}{2}f_{yy}e^2, f_{yz}er, \frac{1}{2}f_{zz}r^2$}\label{section6}
In this section, we use the precise decomposition to handle the terms $\frac{1}{2}f_{yy}e^2, f_{yz}er$. Subsequently, combining the duality relationship established for the terms $\frac{1}{2}f_{yy}e^2, f_{yz}er$ and the rough decomposition established in \autoref{first ad}, we handle the term $\frac{1}{2}f^\e_{zz}r^2$.

Recall that $\bar{y}^\e$ is defined in \eqref{s2}. Then, we obtain that
\begin{align*}
\bar{y}^\e(0) = \dbE\Big[\int_{0}^{T}\Big\{\frac{1}{2}\Gamma^0(s)f^\e_{yy}(s)|e(s)|^2 + \Gamma^0(s)f^\e_{yz}(s)e(s)r(s)+\frac{1}{2}\Gamma^0(s)f^\e_{zz}(s)|r(s)|^2 \Big\}ds\Big].
\end{align*}
We now handle the term $\dbE\Big[\int_{0}^{T}\Big\{\frac{1}{2}\Gamma^0(s)f_{yy}(s)|e(s)|^2 +  \Gamma^0(s)f_{yz}(s)e(s)r(s)\Big\}ds\Big] $. In order to handle these terms, we need to introduce the following adjoint equations:

\begin{equation}\label{aad3}\left\{\begin{aligned}
		&(a)\q P_{5}(s) = \int_{s}^{T} \mathcal{A}_1(\theta') d\theta' - \int_{s}^{T}Q_{5}( \theta')dW_{\theta'}, \q s\in[0, T];\\
		&(b) \q P_{6}(\theta, s) =  \int_{s}^{\theta} \mathcal{A}_2(\theta, \theta') d\theta'  - \int_{s}^{\theta}Q_{6}(\theta, \theta')dW(\theta') ,\q s\in [0, \theta];\\
		&(c) \q P_{7}(\theta, k, s) =  \int_{s}^{\theta\wedge k} \mathcal{A}_3(\theta, k, \theta') d\theta'  - \int_{s}^{\theta\wedge k}Q_{7}(\theta, k, \theta')dW(\theta') ,\q s\in [0, \theta\wedge k],\\
	\end{aligned}\right.\end{equation}
where
\begin{align}
	&\mathcal{A}_1(s) := \frac{1}{2}\Gamma^0(s)^{-1}f_{yy}(s)\lambda^0(s) \lambda_0(s)^\top + \Gamma^0(s)^{-1}f_{yz}(s) \lambda^0(s) \big\{-f_z(s)\lambda^0(s)+\nu^0(s)\big\}^\top,\nonumber\\
	&\mathcal{A}_2(\theta, s):= \Gamma^0(s)^{-1}f_{yy}(s)\lambda^0(s) \dbE_s[\eta^0(\theta)]^\top+ \Gamma^0(s)^{-1}f_{yz}(s)\Big\{\big\{-f_z(s)\lambda^0(s)+\nu^0(s)\big\} \dbE_s[\eta^0(\theta)]^\top\nonumber\\
	&\q\q\q\q\q+\lambda^0(s) \big\{\zeta^0(\theta, s)-f_z(s)\dbE_s[\eta^0(\theta)]\big\}^\top\Big\},\nonumber\\
	&\mathcal{A}_3(\theta, k, s):=\frac{1}{2}\Gamma^0(s)^{-1}f_{yy}(s)\dbE_s[\eta^0(\theta)] \dbE_s[\eta^0(k)]^\top+\Gamma^0(s)^{-1}f_{yz}(s)\big\{\zeta^0(\theta, s)-f_z(s)\dbE_s[\eta^0(\theta)]\big\}\dbE_s[\eta^0(k)]^\top.
\end{align}
Recall that $M_1(\cd), M_2(\cd), M_3(\cd, \cd) $ is defined in \eqref{m}. 
\begin{definition}\label{ad2}
	Replacing $H, M_1(\cd), M_2(\cd), M_3(\cd, \cd) $ with $0, \mathcal{B}_1(\cd), \mathcal{B}_2( \cd), \mathcal{B}_3(\cd, \cd)$ in \eqref{Adjointe}, we denote the solutions $\big(P_2(\cd), Q_2(\cd, \cd)\big)$, $ \big(P_3(\cd, \cd)$, $Q_3(\cd, \cd, \cd)\big)$, $\big(P_{4}(\cd)$, $Q_{4}(\cd, \cd)\big)$ in \eqref{Adjointe} by  
	$\big(P_8(\cd)$, $Q_8(\cd, \cd)\big)$, $\big(P_9(\cd)$, $Q_9(\cd, \cd)\big)$, $\big(P_{10}(\cd, \cd)$, $Q_{10}(\cd, \cd, \cd)\big)$,
 where
 \begin{equation}\label{bb3}
\begin{aligned}
	&\mathcal{A}_4(s):= \Gamma^0(s)^{-1}f_{yz}(s)\Big\{\lambda^0(s)\Big(\lambda^0(s)^\top C(T, s) + \int_{s}^{T}\dbE_s[\eta^0(\theta)]^\top C(\theta, s) d\theta\Big)\Big\},\\
	&\mathcal{A}_5(\theta, s):= \Gamma^0(s)^{-1}f_{yz}(s)\Big\{\dbE_s[\eta^0(\theta)]\Big(\lambda^0(s)^\top C(T, s) + \int_{s}^{T}\dbE_s[\eta^0(\theta)]^\top C(\theta, s) d\theta\Big)\Big\},\\
	&\mathcal{B}_1(s):= 2\big[P_5(s)+P_5(s)^\top\big]A(T, s) + 2\big[Q_5(s)+Q_5(s)^\top\big]C(T, s)+2\int_{s}^{T}\Big\{P_6(k, s)A(k, s) + Q_6(k, s)\\
	& \q\q\q\q C(k, s)\Big\}dk+2\mathcal{A}_4(s),\\
	&
	\mathcal{B}_2(s):= C(T, s)^\top [P_5(s)+P_5(s)^\top]C(T, s) + \int_{s}^{T}\big\{C(T, s)^\top P_6(k, s) C(k, s)+ C(k, s)^\top P_6(k, s)^\top C(T, s) \big\} dk  \\
	&\q\q\q\q+ \int_{s}^{T}\int_{s}^{T}C(\theta, s)^\top [P_7(\theta, k, s)+P_7(k,\theta,  s)^\top]C(k, s)dk d\theta, \\
	&
	\mathcal{B}_3(k, s):= 2P_6(k, s)^\top A(T, s) + 2Q_6(k, s)^\top C(T, s) + 2\int_{s}^{T}\Big\{P_7(k, \theta, s)A(\theta, s)+ Q_7(k, \theta, s)C(\theta, s)  \\
	&\q\q\q\q
	+ P_7(\theta, k, s)^\top  A(\theta, s)+ Q_7(\theta, k, s)^\top C(\theta, s)\Big\}d\theta+2\mathcal{A}_5(k, s). 
\end{aligned}
 \end{equation}
\end{definition}
The following is the well-posedness result for \eqref{aad3}.
\begin{proposition}\label{Ad-uniqueee}
	\raggedright
Let \autoref{A1} hold. Then, we have that, for $\beta >1$, \eqref{aad3} admits a unique solution satisfying $P_5(s), Q_5(s) P_6(\cd, s),Q_6(\cd, s),  P_7(\cd, \cd, s ), Q_7(\cd, \cd, s), P_8(s), Q_8(\cd, s),  P_9( s), Q_9( \cd, s)$, $P_{10}(\cd, s), Q_{10}(\cd, \cd, s) \in \sF_s$ and 
\begin{align*}
&\dbE\Big[\mathop{\sup}_{s\in[0, T]}|P_5(s)|^\beta + \Big(\int_{0}^{T}|Q_5(s)|^2ds\Big)^\frac{\beta}{2}\Big]<\i; \q \dbE\Big[\int_{0}^{T}\mathop{\sup}_{s\in[0, \theta]}|P_6(\theta, s)|^2d\theta + \int_{0}^{T}\int_{0}^{\theta}|Q_6(\theta, s)|^2dsd\theta\Big] <\i;\\
&\dbE\Big[\int_{0}^{T}\int_{0}^{T}\mathop{\sup}_{s\in[0, \theta\wedge k]}|P_7(\theta, k, s)|^2d\theta dk+ \int_{0}^{T}\int_{0}^{T}\int_{0}^{\theta\wedge k}|Q_7(\theta, k, s)|^2dsd\theta dk\Big] <\i;\\
&
\dbE\Big[\int_{0}^{T}|P_8(s)|^2ds + \int_{0}^{T}\int_{0}^{T}|Q_8(\theta, s)|^2dsd\theta\Big] <\i; \\
&
\dbE\Big[\int_{0}^{T}|P_{9}(s)|^2ds + \int_{0}^{T}\int_{0}^{T}|Q_{9}(\theta, s)|^2ds d\theta\Big]<\i;\\
&
 \dbE\Big[\int_{0}^{T}\int_{0}^{T}|P_{10}(k, s)|^2dsdk + \int_{0}^{T}\int_{0}^{T}\int_{0}^{T}|Q_{10}(k, \theta, s)|^2dsd\theta dk\Big] <\i.
\end{align*}
\end{proposition}
\begin{proof}
By \autoref{A1}, H$\ddot{\rm o}$lder's inequality, Tower property, \autoref{ga}, Doob's inequality, and \autoref{aua}, one can deduce that
\begin{align}
&\dbE\Big[\int_{0}^{T}\int_{0}^{T}\Big(\int_{0}^{\theta\wedge k} |\mathcal{A}_3(\theta, k, s)|ds\Big)^2d\theta dk\Big] \nonumber\\
&\leq K\dbE\Big[\mathop{\sup}_{s\in[0, T]}\Gamma^0(s)^{-8}\Big]^\frac{1}{4} \Big\{\dbE\Big[\mathop{\sup}_{s\in [0, T]}\dbE_s\Big(\int_{0}^{T}|\eta^0(\theta)|^2d\theta\Big)^2\Big]^\frac{1}{2}\dbE\Big[\Big(\int_{0}^{T}|\eta^0(\theta)|^2d\theta\Big)^4\Big]^\frac{1}{4}\nonumber\\
&\q\q+\dbE\Big[\Big(\int_{0}^{T}\int_{0}^{T}|\zeta^0(\theta, s)|^2dsd\theta\Big)^2\Big]^\frac{1}{2}\dbE\Big[\Big(\int_{0}^{T}|\eta^0(k)|^2dk\Big)^4\Big]^\frac{1}{4}\Big\}<\i. \label{a3}
\end{align}
Then, for almost everywhere $(\theta, k) \in [0, T]^2$, we have that $\mathcal{A}_3(\theta, k, \cd) \in \mathcal{M}^{1, 2}_\dbF(0, \theta\wedge k; \dbR^{3\times 3}).$ It then follows from Briand--Delyon--Hu--Pardoux--Stoica \cite[Theorem 4.2]{briand2003lp} that 
\begin{align}\label{a4}
\dbE\Big[\mathop{\sup}_{s\in[0, \theta\wedge k]}|P_7(\theta, k, s)|^2+ \int_{0}^{\theta\wedge k}|Q_7(\theta, k, s)|^2ds\Big] \leq K \dbE\Big[\Big(\int_{0}^{\theta\wedge k}|\mathcal{A}_3(\theta, k, s)|ds\Big)^2\Big].
\end{align}
Combining \eqref{a3} and \eqref{a4}, one can deduce that $\dbE\Big[\int_{0}^{T}\int_{0}^{T}\mathop{\sup}_{s\in[0, \theta\wedge k]}|P_7(\theta, k, s)|^2d\theta dk+ \int_{0}^{T}\int_{0}^{T}\int_{0}^{\theta\wedge k}|Q_7(\theta, k, s)|^2dsd\theta dk\Big]<\i,$
which implies the well-posedness result of $\big(P_7(\cd, \cd, \cd), Q_7(\cd, \cd, \cd)\big)$.
Similarly, we obtain the well-posedness result of $(P_5(\cd), Q_5(\cd)), (P_6(\cd, \cd), Q_6(\cd, \cd)), (P_7(\cd, \cd, \cd), Q_7(\cd, \cd, \cd))$. Finally, one can check that $\mathcal{B}_1(\cd) , \mathcal{B}_2(\cd)\in L^2_\dbF(0, T, \dbR^{3\times 3})$ and $\mathcal{B}_3(\cd, \cd) \in   L^2\big(0, T ; L^2_{\dbF}(0, T; \dbR^{3\times 3})\big)$. By a similar deduction of \autoref{Ad-uniquee}, one can obtain the well-posedness result of $\big(P_8(\cd), Q_8(\cd)\big), \big(P_9(\cd), Q_9(\cd)\big), \big(P_{10}(\cd, \cd), Q_{10}(\cd, \cd)\big)$.

\end{proof}

\begin{lemma}\label{important}
Let \autoref{A1} hold. Then, we have that
\begin{align}\label{fyy}
\dbE\Big[\int_{0}^{T}\Big\{\frac{1}{2}\Gamma^0(s)f^\e_{yy}(s)|e(s)|^2 +  \Gamma^0(s)f^\e_{yz}(s)e(s)r(s)\Big\}ds\Big] = \dbE\Big[\int_{0}^{T}I_1(s)ds\Big]+o(\e), 
\end{align}
where
\begin{align}
&I_1(s):=D(T, s)^\top P_5(s)D(T, s)
+
\int_{s}^{T}D(T, s)^\top P_6(k, s) D(k,s)dk + \int_{s}^{T}\int_{s}^{T}D(\theta, s)^\top P_7(\theta, k, s)D(k, s)d\theta dk  \nonumber \\
&+\frac{1}{2}\Big\{ \int_{s}^{T}\{D(T, s)^\top P_8(k)^\top D(k, s)+D(k, s)^\top P_8(k) D(T, s)\}dk +\int_{s}^{T}D(\theta, s)^\top P_{9}(\theta )D(\theta, s)d\theta\nonumber\\
&+ \int_{s}^{T}\int_{s}^{T} D(k, s)^\top P_{10}( \theta, k) D(\theta, s)dk d\theta\Big\}.
\end{align}
\end{lemma}
\begin{proof}
By H$\ddot{\rm o}$lder's inequality, \autoref{ro est}, \autoref{e2p},  and Dominated convergence theorem, one can deduce that
\begin{align*}
	&\dbE\Big[\int_{0}^{T}\Big\{\frac{1}{2}\Gamma^0(s)f^\e_{yy}(s)|e(s)|^2 +  \Gamma^0(s)f^\e_{yz}(s)e(s)r(s)\Big\}ds\Big]\\ &= \dbE\Big[\int_{0}^{T}\Big\{\frac{1}{2}\Gamma^0(s)f_{yy}(s)|e(s)|^2 +  \Gamma^0(s)f_{yz}(s)e(s)r(s)\Big\}ds\Big] +o(\e).
\end{align*}
It then follows from
\begin{align}
	&e(s) - e_2(s) = \Gamma^0(s)^{-1}\Big\{\lambda^0(s)^\top\mathcal{X}_1(T, s) + \int_{s}^{T}\dbE_s[\eta^0(\theta)]^\top\mathcal{X}_1(\theta, s)d\theta\Big\},\label{e}\\
	&
	r(s)-\Gamma^0(s)^{-1}\big(r^0_{12}(s) + r_2^0(s)\big)+f_z(s)e_2(s)= \Gamma^0(s)^{-1}\Big[\lambda^0(s)^\top C(T, s)X_1(s)+ \{-f_z(s)\lambda^0(s)^\top +\nu^0(s)^\top\} \nonumber\\
	&\q\q\q\q
	\mathcal{X}_1(T, s)+ \int_{s}^{T}\Big\{\dbE_s[\eta^0(\theta)]^\top C(\theta, s) X_1(s)+\big\{-f_z(s)\dbE_s[\eta^0(\theta)]^\top +\zeta^0(\theta, s)^\top\big\}\mathcal{X}_1( \theta, s)\Big\}d\theta\Big],
\end{align}
that
\begin{align}\label{iii}
&\dbE\Big[\int_{0}^{T}\Big\{\frac{1}{2}\Gamma^0(s)f_{yy}(s)|e(s)|^2 +  \Gamma^0(s)f^\e_{yz}(s)e(s)r(s)\Big\} ds\Big]
= \dbE\Big[\int_{0}^{T}\mathcal{X}_1(T, s)^\top \mathcal{A}_1(s) \mathcal{X}_1(T, s)ds\Big] \nonumber\\
& + \dbE\Big[\int_{0}^{T}\int_{s}^{T}\mathcal{X}_1(T, s)^\top 
%
\mathcal{A}_2(\theta, s) \mathcal{X}_1(\theta, s)d\theta ds\Big]
+ \dbE\Big[\int_{0}^{T}\int_{s}^{T}\int_{s}^{T}\mathcal{X}_1(\theta, s)^\top \mathcal{A}_3(\theta, k ,s) \mathcal{X}_1(k, s)d\theta dk ds\Big]\nonumber \\
&
+\dbE\Big[\int_{0}^{T}\mathcal{X}_1(T, s)^\top\mathcal{A}_4( s)X_1(s) ds\Big]
+\dbE\Big[\int_{0}^{T}\int_{s}^{T}\mathcal{X}_1(\theta, s)^\top \mathcal{A}_5(\theta, s)X_1(s)d\theta ds\Big] 
+o(\e).
\end{align}

{\bf Step 1:}  let us first eliminate the terms $\mathcal{X}_1(T, s)^\top [...] \mathcal{X}_1(T, s)$, $\mathcal{X}_1(T, s)^\top [...] \mathcal{X}_1(\theta, s)$, and $\mathcal{X}_1(\theta, s)[...]\mathcal{X}_1(k, s)$ by using \eqref{aad3}.
\\
Applying It$\hat{\rm o}$'s formula to $\big\{\mathcal{X}_1(T, s)P_5(s)\mathcal{X}_1(T, s)\big\}_{s\in [0, T]}$, $\big\{\mathcal{X}_1(T, s)P_6(\theta, s)\mathcal{X}_1(\theta, s)\big\}_{s\in [0, \theta]}$, and  $\big\{\mathcal{X}_1(\theta, s)^\top P_7(\theta, k, s)\mathcal{X}_1(k, s)\big\}_{s\in[0, \theta\wedge k]}$, one can deduce that
\begin{align}\label{a1}
&\dbE\Big[\int_{0}^{T}\Big\{\frac{1}{2}\Gamma^0(s)f_{yy}(s)|e(s)|^2 +  \Gamma^0(s)f_{yz}(s)e(s)r(s)\Big\}ds\Big] = \dbE\Big[\frac{1}{2}\int_{0}^{T}\mathcal{X}_1(T, s)^\top \mathcal{B}_1(s)X_1(s)ds\Big] +\dbE\Big[\frac{1}{2}\nonumber\\
&\int_{0}^{T}X_1(s)^\top \mathcal{B}_2(s)X_1(s)ds\Big]+
\dbE\Big[\frac{1}{2}\int_{0}^{T}\int_{s}^{T}\mathcal{X}_1(k, s)^\top\mathcal{B}_3(k, s)
X_1(s)dk ds \Big] + \dbE\Big[\int_{0}^{T}D(T, s)P_5(s)D(T, s)ds\Big]\nonumber
\\
& +
\dbE\Big[\int_{0}^{T}\int_{s}^{T}D(T, s)^\top P_6(k, s)
 D(k,s)dk ds\Big]+ \dbE\Big[\int_{0}^{T}\int_{s}^{T}\int_{s}^{T}D(\theta, s)^\top P_7(\theta, k, s)D(k, s)d\theta dk ds\Big] +o(\e),
\end{align}

{\bf Step 2:}
by a similar deduction as in step 2 of \autoref{first ad},  one can deduce that \eqref{fyy} holds.

\end{proof}
Finally, we handle the term $\dbE\Big[\int_{0}^{T}\Big\{ \frac{1}{2}\Gamma^0(s)f^\e_{zz}(s)|r(s)|^2 \Big\}ds\Big]$. We now introduce the adjoint equations. Since the deduction of the well-posedness result is similar to \autoref{Ad-uniqueee}, we would like to omit them.
\begin{definition}
(i) \q \begin{align}\label{111}
	\varsigma(t) = \frac{1}{2}\|f_{zz}\|_\i \Gamma^0(T)\bar{H}^\top \bar{H} - \int_{t}^{T}\vartheta(s)dW(s), \q t\in[0, T].
\end{align}
(ii)\q Let $\big(\tilde{P}_5(\cd)$,$ \tilde{Q}_5(\cd)\big)$,$ \big(\tilde{P}_6(\cd, \cd)$,$ \tilde{Q}_6(\cd, \cd)\big) $,$ \big(\tilde{P}_7(\cd, \cd, \cd)$,$ \tilde{Q}_7(\cd, \cd, \cd)\big)$ be the unique solutions obtained by replacing $\mathcal{A}_1$,$\mathcal{A}_2$,$\mathcal{A}_3$ with $\tilde{\mathcal{A}}_1$,$ \tilde{\mathcal{A}}_2$, $\tilde{\mathcal{A}}_3$ in \eqref{aad3}. Subsequently, we denote $ \tilde{\mathcal{B}}_1$, $ \tilde{\mathcal{B}}_2$, $ \tilde{\mathcal{B}}_3$ by replacing $\mathcal{A}_4$, $ \mathcal{A}_5$, $\big(P_5(\cd)$, $ Q_5(\cd)\big)$,$ \big(P_6(\cd), Q_6(\cd)\big)$,$\big(P_7(\cd), Q_7(\cd)\big)$  with $\tilde{\mathcal{A}}_4,\tilde{\mathcal{A}}_5, \big(\tilde{P}_5(\cd)$,$ \tilde{Q}_5(\cd)\big)$,$ \big(\tilde{P}_6(\cd, \cd), \tilde{Q}_6(\cd, \cd)\big) , \big(\tilde{P}_7(\cd, \cd, \cd), \tilde{Q}_7(\cd, \cd, \cd)\big)$ in \eqref{bb3}. Furthermore, let $\big(\tilde{P}_8(\cd), \tilde{Q}_8(\cd)\big), \big(\tilde{P}_{9}(\cd), \tilde{Q}_{9}(\cd, \cd)\big), \big(\tilde{P}_{10}(\cd, \cd), \tilde{Q}_{10}(\cd, \cd, \cd)\big)$ be the unique solutions by replacing $\mathcal{B}_1, \mathcal{B}_2,\mathcal{B}_3$ with $\tilde{\mathcal{B}}_1, \tilde{\mathcal{B}}_2+2\tilde{A}_6, \tilde{\mathcal{B}}_3$ in \autoref{ad2},
where 
\begin{equation}\label{aa3}
\begin{aligned}
	&\tilde{\mathcal{A}}_1(s) :=  \frac{1}{2}\|f_{zz}\|_\i \Gamma^0(s)^{-1}f_y(s)\lambda^0(s) \lambda^0(s)^\top,\\
	&\tilde{\mathcal{A}}_2(\theta, s):= \|f_{zz}\|_\i \Gamma^0(s)^{-1}f_y(s)\lambda^0(s) \dbE_s[\eta^0(\theta)]^\top,
	\q 
	\tilde{\mathcal{A}}_3(\theta, k, s):= \frac{1}{2}\|f_{zz}\|_\i \Gamma^0(s)^{-1}f_y(s) \dbE_s[\eta^0(\theta)] \dbE_s[\eta^0(k)]^\top,
	\\
	&\tilde{\mathcal{A}}_4(s) := 2\varsigma(s) A(T, s) + 2\vartheta(s)C(T, s)+\|f_{zz}\|_\i \lambda^0(s) \bar{F}(s),\q
	\tilde{\mathcal{A}}_5(\theta, s):= \|f_{zz}\|_\i \dbE_s[\eta^0(\theta)] \bar{F}(s)
	,\\
	&
	\tilde{\mathcal{A}}_6(s):=C(T, s)^\top\varsigma(s)	C(T, s).
\end{aligned}\end{equation}
\end{definition}
\begin{lemma}\label{iim}
Let \autoref{A1} hold. Then, we have that
\begin{align}\label{fz}
&\dbE\Big[\int_{0}^{T}\Big\{ \frac{1}{2}\Gamma^0(s)f^\e_{zz}(s)|r(s)|^2 \Big\}ds\Big]\leq \dbE\Big[\int_{0}^{T}I_2(s)ds\Big]+o(\e),
\end{align}
where
\begin{align}
&I_2(s) := D(T, s)^\top\big\{\varsigma(s)+\tilde{P}_5(s)\big\}D(T, s) +
\int_{s}^{T}D(T, s)^\top \tilde{P}_6(k, s) D(k,s)dk + \int_{s}^{T}\int_{s}^{T}D(\theta, s)^\top \tilde{P}_7(\theta, k, s) \nonumber \\
&D(k, s)d\theta dk + \frac{1}{2}\Big\{\int_{s}^{T}\{D(T, s)^\top \tilde{P}_8(k)^\top D(k, s)+D(k, s)^\top \tilde{P}_8(k)D(T, s)\}dk +\int_{s}^{T}D(\theta, s)^\top \tilde{P}_{9}(\theta)D(\theta, s)d\theta\nonumber\\
&+ \int_{s}^{T}\int_{s}^{T} D(k, s)^\top \tilde{P}_{10}( \theta, k) D(\theta, s)dk d\theta\Big\}.
\end{align}
\end{lemma}
\begin{proof}
Recall that $\big(e(s), r(s)\big)$ is defined in \eqref{ee}. Applying It$\hat{\rm o}$'s formula to $\big\{\frac{1}{2}\Gamma^0(s)\big(e(s)\big)^2\big\}_{s\in[0, T]}$, one can deduce that
\begin{align}
&\dbE\Big[\int_{0}^{T}\Big\{ \frac{1}{2}\Gamma^0(s)f^\e_{zz}(s)|r(s)|^2 \Big\}ds\Big]\leq \|f_{zz}\|_\i\dbE\Big[\int_{0}^{T}\Big\{ \frac{1}{2}\Gamma^0(s)|r(s)|^2 \Big\}ds\Big]\nonumber \\
&= \frac{1}{2}\|f_{zz}\|_\i \Big\{\dbE\Big[X_1(T)^\top\Gamma^0(T)\bar{H}^\top \bar{H}X_1(T)+\int_{0}^{T}\Gamma^0(s)f_y(s)|e(s)|^2 + 2\Gamma^0(s)e(s)\bar{F}(s)X_1(s)\Big]\Big\}+o(\e)\nonumber\\
&
=\frac{1}{2}\|f_{zz}\|_\i \Big\{\dbE\Big[X_1(T)^\top\Gamma^0(T)\bar{H}^\top \bar{H}X_1(T)+\int_{0}^{T}\Gamma^0(s)f_y(s)|e(s)-e_2(s)|^2 \nonumber\\
&\q\q+ 2\Gamma^0(s)[e(s)-e_2(s)]\bar{F}(s)X_1(s)\Big]\Big\}+o(\e).\label{s}
\end{align}
We now  handle the term $X_1(T)^\top [...]X_1(T)$.
Applying It$\hat{\rm o}$'s formula to $\Big\{\mathcal{X}_1(T, s)^\top \varsigma(s)\mathcal{X}_1(T, s)\Big\}_{s\in[0, T]}$ and substituting \eqref{e} to \eqref{s}, one can deduce that
\begin{align}\label{iv}
&\dbE\Big[\int_{0}^{T}\Big\{ \frac{1}{2}\Gamma^0(s)f^\e_{zz}(s)|r(s)|^2 \Big\}ds\Big]\leq \dbE\Big[\int_{0}^{T}\mathcal{X}_1(T, s)^\top \tilde{\mathcal{A}}_1(s) \mathcal{X}_1(T, s)ds\Big] \nonumber\\
& + \dbE\Big[\int_{0}^{T}\int_{s}^{T}\mathcal{X}_1(T, s)^\top 
%
\tilde{\mathcal{A}}_2(\theta, s) \mathcal{X}_1(\theta, s)d\theta ds\Big]
+ \dbE\Big[\int_{0}^{T}\int_{s}^{T}\int_{s}^{T}\mathcal{X}_1(\theta, s)^\top \tilde{\mathcal{A}}_3(\theta, k ,s) \mathcal{X}_1(k, s)d\theta dk ds\Big]\nonumber\\
&
+\dbE\Big[\int_{0}^{T}\mathcal{X}_1(T, s)^\top\tilde{\mathcal{A}}_4( s)X_1(s) ds\Big]
+\dbE\Big[\int_{0}^{T}\int_{s}^{T}\mathcal{X}_1(\theta, s)^\top \tilde{\mathcal{A}_5}(\theta, s)X_1(s)d\theta ds\Big]\nonumber \\
&+ \dbE\Big[\int_{0}^{T}X_1(s)^\top \tilde{\mathcal{A}}_6(s) X_1(s)ds\Big]  + \dbE\Big[\int_{0}^{T}D(T, s)^\top \varsigma(s)D(T, s)ds\Big]
+o(\e),
\end{align}
where $\tilde{A}_1, \tilde{A}_2, \tilde{A}_3, \tilde{A}_4, \tilde{A}_5, \tilde{A}_6$ are defined in \eqref{aa3}.
In view of \eqref{iii} and \eqref{iv},  by a similar deduction of the proof of \autoref{important}, we obtain \eqref{fz}.
\end{proof}

\section{Maximum principle}\label{section7}
In this section, we give a maximum principle for {\bf Problem  (O)}.
Define 
\begin{align*}
&\bar{P}_1(s):= P_1(s) +2P_5(s)+2\tilde{P}_5(s)+2\varsigma(s), \q \bar{P}_2(\theta, s) := \dbE_s[P_2(\theta) +P_6(\theta, s)^\top+ \tilde{P}_6(\theta, s)^\top + P_8(\theta)+\tilde{P}_8(\theta)] , \\
& \bar{P}_3(\theta, s) :=  \dbE_s[P_3(\theta) + P_9(\theta) + \tilde{P}_9(\theta)],\\
& \bar{P}_4(\theta, \theta', s) :=  \dbE_s[P_4(\theta, \theta') + 2P_7(  \theta', \theta, s)+ 2\tilde{P}_7(\theta', \theta,  s)+ P_{10}(\theta, \theta') + \tilde{P}_{10}(\theta, \theta')] .
\end{align*}
and 
\begin{equation}\label{mp}
	\begin{aligned}
\mathcal{P}(s) &:= \bar{P}_1^{11}(s) + \big\{ \bar{P}_1^{12}(s)+\bar{P}_1^{21}(s)+\bar{P}_1^{22}(s)\big\}1_{(0, T-\delta)}(s)\\
& +\int_{s}^{T} \big\{\bar{P}_2^{11}(\theta, s)^\top + \bar{P}_2^{11}(\theta, s) \big\}d\theta+\int_{s}^{T} \big\{\bar{P}_2^{12}(\theta, s)^\top + \bar{P}_2^{12}(\theta, s)\big\}d\theta 1_{(0, T-\delta)}(s) \\
&
+ \int_{s+\delta}^{T}\big\{\bar{P}_2^{21}(\theta, s)^\top +\bar{P}_2^{21}(\theta, s) + \bar{P}_2^{22}(\theta, s)^\top + \bar{P}^{22}_2(\theta, s)\big\}d\theta 1_{(0, T-\delta)}(s)\\
&
+\int_{s}^{T}\bar{P}_3^{11}(\theta, s)d\theta + \int_{s+\delta}^{T}\big\{\bar{P}_3^{21}(\theta, s)+\bar{P}_3^{12}(\theta, s)+\bar{P}_3^{22}(\theta, s)\big\}d\theta 1_{(0, T-\delta)}(s)\\
&
+\int_{s}^{T}\int_{s}^{T}\bar{P}_4^{11}(\theta, \theta', s)d\theta d\theta' + \Big\{\int_{s}^{T}\int_{s+\delta}^{T}\bar{P}_4^{12}(\theta, \theta',s)d\theta d\theta' \\
&+ \int_{s+\delta}^{T}\int_{s}^{T}\bar{P}_4^{21}(\theta, \theta',s)d\theta d\theta' + \int_{s+\delta}^{T}\int_{s+\delta}^{T}\bar{P}_4^{22}(\theta, \theta',s)d\theta d\theta'\Big\}1_{(0, T-\delta)}(s),\\
&\tilde{p}(s):= \Gamma^0(s)^{-1}\{\lambda^{00}(s)+\int_{s}^{T}\dbE_s[\eta^{00}(\theta)]d\theta + \lambda^{01}(s)1_{(0, T-\delta)}(s)+\int_{s+\delta}^{T}\dbE_s[\eta^{01}(\theta)]d\theta1_{(0, T-\delta)}(s)\},
	\end{aligned}
\end{equation}
where
\begin{equation}
\bar{P}_1(\cd) :=
\begin{bmatrix}
\bar{P}_1^{11}(\cd)& \bar{P}_1^{12}(\cd)&\bar{P}_1^{13}(\cd)\\
\bar{P}_1^{21}(\cd)&\bar{P}_1^{22}(\cd)&\bar{P}_1^{23}(\cd)\\
\bar{P}_1^{31}(\cd)&\bar{P}_1^{32}(\cd)&\bar{P}_1^{33}(\cd)
\end{bmatrix},\q
\bar{P}_2(\cd, \cd):= 
\begin{bmatrix}
\bar{P}_2^{11}(\cd, \cd)& \bar{P}_2^{12}(\cd, \cd)&\bar{P}_2^{13}(\cd, \cd)\\
\bar{P}_2^{21}(\cd, \cd)&\bar{P}_2^{22}(\cd, \cd)&\bar{P}_2^{23}(\cd, \cd)\\
\bar{P}_2^{31}(\cd, \cd)&\bar{P}_2^{32}(\cd, \cd)&\bar{P}_2^{33}(\cd, \cd)
\end{bmatrix}
\end{equation}
\begin{equation}
\bar{P}_3(\cd, \cd):=
\begin{bmatrix}
\bar{P}_3^{11}(\cd, \cd)& \bar{P}_3^{12}(\cd, \cd)&\bar{P}_3^{13}(\cd, \cd)\\
\bar{P}_3^{21}(\cd, \cd)&\bar{P}_3^{22}(\cd, \cd)&\bar{P}_3^{23}(\cd, \cd)\\
\bar{P}_3^{31}(\cd, \cd)&\bar{P}_3^{32}(\cd, \cd)&\bar{P}_3^{33}(\cd, \cd)
\end{bmatrix}\q
\bar{P}_4(\cd, \cd, \cd):=
\begin{bmatrix}
\bar{P}_4^{11}(\cd, \cd, \cd)& \bar{P}_4^{12}(\cd, \cd, \cd)&\bar{P}_4^{13}(\cd, \cd, \cd)\\
\bar{P}_4^{21}(\cd, \cd, \cd)&\bar{P}_4^{22}(\cd, \cd, \cd)&\bar{P}_4^{23}(\cd, \cd, \cd)\\
\bar{P}_4^{31}(\cd, \cd, \cd)&\bar{P}_4^{32}(\cd, \cd, \cd)&\bar{P}_4^{33}(\cd, \cd, \cd)
\end{bmatrix},
\end{equation}
\begin{equation}
\lambda^0(\cd):= 
\begin{bmatrix}
&\lambda^{00}(\cd) \\
&\lambda^{01}(\cd) \\
&\lambda^{02}(\cd) 
\end{bmatrix},\q
\eta^0(\cd):= 
\begin{bmatrix}
	&\eta^{00}(\cd) \\
	&\eta^{01}(\cd) \\
	&\eta^{02}(\cd) 
\end{bmatrix}
\end{equation}

Based on the above preparation, now we can state the general maximum principle. Define the Hamiltonian function: $\mathcal{H}: [0, T] \times \Om \times \dbR^9 \times U \times U$ by
\begin{equation}\label{Hamilton}
\begin{aligned}
&\mathcal{H}(t,x,x_\delta, \tilde{x}, y,z,p,q, \tilde{p}, \mathcal{P}, u, \mu) := pb(t, x, x_\delta, \tilde{x}, u, \mu)+q\sigma(t, x, x_\delta, \tilde{x}, u, \mu) \\
&+\Gamma^0(t)f\big(t, x, x_\delta, \tilde{x}, y, z + \tilde{p}\big(\sigma(t, x, x_\delta, \tilde{x}, u, \mu)- \sigma(t)\big), u, \mu\big) - \Gamma^0(t)f\big(t, x, x_\delta, \tilde{x}, y, z + \tilde{p}\big(\sigma(t, x, x_\delta, \tilde{x}, u, \mu)\\
&- \sigma(t)\big), u^*(t), \mu^*(t)\big) 
 +  \frac{1}{2}\Big(\sigma(t,x,x_\delta, \tilde{x},  u, \mu) - \sigma(t)\Big)^\top \mathcal{P}\Big(\sigma(t,x,x_\delta, \tilde{x},  u, \mu) - \sigma(t)\Big).
\end{aligned}\end{equation}
Now we present the stochastic maximum principle. 
%
\begin{theorem}\label{SMP-thm}
Suppose that \autoref{A1}  holds. Let $u^*(\cd) \in \mathcal{U}_{ad}$ be optimal control and $(x^*(\cd), y^*(\cd), z^*(\cd))$ be the corresponding state trajectories of \eqref{2.1}. Then the following stochastic maximum principle holds:
\begin{align}\label{SMP}
\Delta \mathcal{H}(t) + \dbE_t[\Delta\tilde{\mathcal{H}}(t+\delta)1_{[0, T-\delta]}(t)] \geq 0, \q \forall u\in U, \q a.e., \q a.s.,
\end{align}
where
\begin{align*}	
&\Delta \mathcal{H}(t) = \mathcal{H}(t,x^*(t), x_\delta^*(t), \tilde{x}^*(t), y^*(t), z^*(t), p(t), q(t), \tilde{p}(t),\mathcal{P}(t), u, \mu^*(t) ) \\
&\q\q\q\q- \mathcal{H}(t,x^*(t), x_\delta^*(t), \tilde{x}^*(t) , y^*(t), z^*(t), p(t), q(t), \tilde{p}(t), \mathcal{P}(t), u^*(t), \mu^*(t) ), \\
&\Delta \tilde{\mathcal{H}}(t) = \mathcal{H}(t,x^*(t), x_\delta^*(t), \tilde{x}^*(t),y^*(t), z^*(t), p(t), q(t), \tilde{p}(t), \mathcal{P}(t), u^*(t), u ) \\
&\q\q\q\q- \mathcal{H}(t,x^*(t), x_\delta^*(t), \tilde{x}^*(t),y^*(t), z^*(t), p(t), q(t), \tilde{p}(t), \mathcal{P}(t), u^*(t), \mu^*(t) ),
\end{align*}
$\big(p(\cd), q(\cd)\big), \tilde{p}(\cd)$,  $\mathcal{P}(\cd)$ are defined in \eqref{pq} and \eqref{mp}, respectively.
\end{theorem}
\begin{proof}
In view of \autoref{last1}, \autoref{important}, and \autoref{iim}, we obtain that 
\begin{align*}
0 \leq &J(u^{\e}(\cd)) - J(u^*(\cd)) = y^{\e}(0) - y^*(0) = \hat{y}(0) + \bar{y}(0) + o(\e) \leq \dbE\Big[\int_{0}^{T}\big\{\Gamma^0(t)I(t) +I_1(t) + I_2(t)\big\}dt\Big]+ o(\e) \\
&= \dbE\Big[\int_{0}^{T}\Big\{\Delta \mathcal{H}(t) + \dbE_t[\Delta\tilde{\mathcal{H}}(t+\delta)1_{[0, T-\delta]}(t)]\Big\}1_{[t_0, t_0+\e]}(t)dt\Big] + o(\e).
\end{align*}
Then, by a similar deduction of Meng and Shi \cite[Theorem 4.1]{Meng2021}, we obtain \eqref{SMP}.
\end{proof}
{\bf Statement:} During the preparation of this manuscript, the authors used ChatGPT solely for language polishing, grammar checking, and assistance in correcting errors. The author takes full responsibility for the content of the manuscript.

\end{document}